\documentclass[11pt, a4paper,reqno]{amsart}
\usepackage[utf8]{inputenc}
\usepackage{enumerate}
\usepackage{amsmath, amssymb,amsthm}
\usepackage{mathtools}
\usepackage{leftidx}
\numberwithin{equation}{section}
\usepackage[left=2.5cm, right=2.5cm,bottom=4cm]{geometry}
\usepackage{hyperref}
\usepackage{xcolor}
\definecolor{my-black}{rgb}{0,0,0}
\definecolor{my-blue}{rgb}{0,0,0.8}
\definecolor{my-red}{rgb}{0.8,0,0} 
\definecolor{my-green}{rgb}{0,0.5,0}
\hypersetup{colorlinks,
	linkcolor	= {my-blue},
	citecolor	= {my-red},
	urlcolor		= {my-black}
}
\usepackage{graphicx}
\usepackage{tikz-cd}

\theoremstyle{plain} %italic

\newtheorem{lemma}{Lemma}[section]
\newtheorem{theorem}{Theorem}

\newtheorem{proposition}{Proposition}[section]

\newtheorem*{theorem*}{Theorem} %to avoid numbering

\theoremstyle{definition} %non-italic
\newtheorem{remark}{Remark}[section]

\theoremstyle{remark}
\newtheorem*{remark-non}{Remark}

\DeclareMathOperator{\supp}{supp}
\DeclareMathOperator{\vol}{vol}

\newcommand{\R}{\mathbb{R}}

\newcommand{\Ss}{\mathcal{S}}

\renewcommand{\d}{\mathrm{d}}

\newcommand{\T}{\mathcal{T}}
\newcommand{\eps}{\varepsilon}

\title[Uniform stability for Pr\'ekopa-Leindler]{Improved stability versions \\
of the Pr\'ekopa--Leindler inequality}
\author{Alessio Figalli and Jo{\~a}o P.G. Ramos}

\begin{document}

\begin{abstract} We consider the problem of stability for the Pr\'ekopa--Leindler inequality. Exploiting properties of the transport map between radially decreasing functions and a suitable functional version of the trace inequality, we obtain a \emph{uniform} stability exponent for the Pr\'ekopa--Leindler inequality.

Our result yields an exponent not only uniform in the dimension but also in the log-concavity parameter $\tau = \min(\lambda,1-\lambda)$ associated with its respective version of the Pr\'ekopa--Leindler inequality.  
As a further application of our methods, we prove a \emph{sharp} stability result for \emph{log-concave} functions in dimension 1, which also extends to a sharp stability result for log-concave \emph{radial} functions in higher dimensions. 

\bigskip

\begin{center}
  {\it In honor of R. T. Rockafellar, for his 90th birthday.}
\end{center}

\
\end{abstract}

\maketitle

\section{Introduction} 

 In this paper, we deal with stability versions of the Pr\'ekopa--Leindler inequality. This inequality asserts that, for measurable non-negative functions $h,f,g:\R^n \to \R_+$ which satisfy 
\begin{equation}\label{eq:PL-condition}
h\left( \lambda x + (1-\lambda) y\right) \ge f(x)^{\lambda} g(y)^{1-\lambda}, \qquad \forall \, x,y\ \in \R^n, 
\end{equation}
then 
\begin{equation}\label{eq:PL-conclusion}
\int_{\R^n} h \ge \left( \int_{\R^n} f\right)^{\lambda} \left( \int_{\R^n} g \right)^{1-\lambda},
\end{equation}
with equality if and only if $h$ is log-concave, and for some $x_0 \in \R^n, a>0,$ we have $$f(x) = a^{-\lambda} h(x-\lambda x_0), \quad g(x) = a^{1-\lambda} h(x+(1-\lambda)x_0).$$ Such an inequality is of pivotal importance in areas such as convex geometry and high-dimensional probability, and is intimately related to the \emph{Brunn--Minkowski} inequality, which asserts that, for $A,B \subset \R^n$ measurable sets with $A+B := \{a+b, a \in A, b \in B\}$ measurable, we have
\begin{equation}\label{eq:B-M} 
|A+B|^{1/n} \ge |A|^{1/n}+|B|^{1/n},
\end{equation}
with equality if and only if $A,B$ are \emph{both} convex sets, and $A$ is a translated scaled copy of $B.$ Indeed, the Pr\'ekopa--Leindler is not only a functional version of Brunn--Minkowski, but it also \emph{implies} the latter through a scaling argument. 

When it comes to the question of \emph{stability} for both the Pr\'ekopa--Leindler and Brunn--Minkowski inequalities, the matters are much more subtle: indeed, in the Brunn--Minkowski case, the first work in that direction was achieved by the first author, Maggi, and Pratelli \cite{Figalli-Maggi-Pratelli-2}, which proved a sharp stability version when the sets $A,B$ satisfying almost equality in \eqref{eq:B-M} are already convex. In spite of this initial contribution, a proof of stability of Brunn--Minkowski, even in the non-sharp case, would only be addressed in the works of Christ in a series of papers \cite{Christ-1,Christ-2,Christ-3}, with an explicit exponent of stability only being achieved later by the works of Jerison and the first author \cite{Figalli-Jerison-1}. Other relevant works addressing the question of stability for that inequality are \cite{Figalli-Jerison-2, van-Hintum-Spink-Tiba-1, Kolesnikov-Milman}. 

A major breakthrough in that regard is the settling of the sharp stability for the Brunn--Minkowski in the general case. Indeed, for $n=2,$ a proof was found by van Hintum, Spink, and Tiba \cite{van-Hintum-Spink-Tiba-2}, and, for the $n\ge 3$ case, in the recent paper by the first author, van Hintum, and Tiba \cite{Figalli-van-Hintum-Tiba}. 

On the other hand, stability results for the Pr\'ekopa--Leindler inequality are comparatively significantly less abundant at the moment. We mention initial works by B\"or\"oczky and Ball \cite{Boroczky-Ball-1,Boroczky-Ball-2} in the one-dimensional log-concave case, the recent contribution by B\"or\"oczky and De \cite{Boroczky-De} in the higher-dimensional log-concave setting, and the contributions by Rossi and Salani \cite{Rossi-Salani}, which deal with other non-degenerate cases of the Borell--Brascamp--Lieb inequality. 

Stability for the case of measurable functions, even with a non-sharp exponent, remained open until the recent work of B\"or\"oczky and the authors of this manuscript \cite{Boroczky-Figalli-Ramos}, where it was shown that if $h,f,g$ are as in the statement of the Pr\'ekopa--Leindler inequality, such that
\begin{equation}\label{eq:PL-almost-eq}
\int_{\R^n} h(x) \, \d x \le (1+\varepsilon) \left( \int_{\R^n} f(x) \, \d x\right)^{\lambda} \left( \int_{\R^n} g(x) \, \d x \right)^{1-\lambda}, 
\end{equation}
then there is a log-concave function $\tilde{h}:\R^n\to \R_+$ and $a > 0, x_0 \in \R^n$  such that 
\begin{align*}
\int_{\R^n} |f(x) - a^{-\lambda} \tilde{h}(x-\lambda x_0)| \, \d x & \le C(\tau) \varepsilon^{\alpha_n(\tau)} \int_{\R^n} f(x) \, \d x, \cr 
\int_{\R^n} |g(x) - a^{1-\lambda} \tilde{h}(x+(1-\lambda)x_0)| \, \d x & \le C(\tau) \varepsilon^{\alpha_n(\tau)} \int_{\R^n} g(x) \, \d x, \cr
\int_{\R^n} |h-\tilde{h}| & \le C(\tau) \varepsilon^{\alpha_n(\tau)} \int_{\R^n} h(x) \, \d x.
\end{align*}
in all dimensions $n \ge 1,$ with a computable exponent $\alpha_n(\tau)>0$. It is important to note that the exponent $\alpha_n(\tau)$ obtained in \cite{Boroczky-Figalli-Ramos} is highly dependent on the dimension and on $\tau = \min(\lambda,1-\lambda).$

The purpose of this manuscript is to improve the previous result in different contexts. In fact, the first result we present here is a proof of a \emph{uniform} stability exponent for the Pr\'ekopa--Leindler inequality for general functions in \emph{any} dimension, for \emph{any} $\lambda \in (0,1)$: 

\begin{theorem}\label{thm:uniform-high-d} There is an absolute constant $c_0 > 0$ such that the following holds. 
Let $h,f,g:\R^n \to \R$ be non-negative measurable functions satisfying \eqref{eq:PL-condition}. Suppose, additionally, that they satisfy \eqref{eq:PL-almost-eq}. Then there are $a>0, x_0 \in \R^n$, and a log-concave function $\tilde{h}:\R^n\to\R_+$, such that 
\begin{align*} 
\int_{\R^n} |f(x) - a^{-\lambda} \tilde{h}(x- \lambda x_0)| \, \d x \le C_n(\tau) \varepsilon^{c_0} \int_{\R^n} f(x) \, \d x, \cr 
\int_{\R^n} |g(x) - a^{1-\lambda} \tilde{h}(x+(1-\lambda)x_0)| \, \d x \le C_n(\tau) \varepsilon^{c_0} \int_{\R^n} g(x) \, \d x, \cr 
\int_{\R^n} |h(x) - \tilde{h}(x)| \, \d x \le C_n(\tau) \varepsilon^{c_0} \int_{\R^n} h(x) \, \d x. 
\end{align*} 
\end{theorem}

\begin{remark}
The proof of Theorem \ref{thm:uniform-high-d} may be divided into two main parts. The first is a new proof of the one-dimensional result in  \cite{Boroczky-Figalli-Ramos}. For that, instead of resorting to a four-point inequality as done in that manuscript, we use an idea similar to the one used in  \cite{Boroczky-Figalli-Ramos} but for higher-dimensional functions. In effective terms, by a similar truncation argument, we arrive at the fact that, if we wish to use the stability results for the Brunn--Minkowski inequality available in dimension 2, we only need to prove that the distribution functions of $f$ and $g$ as in the statement are close in $L^1,$ upon appropriate scaling and translating. 

It is equivalent to work with even, radially decreasing functions in dimension one. In order to prove a uniform stability for those, we consider the transport map between them. Upon appropriately cutting both $f$ and $g$ at a well-controlled level set, we see that the derivative of the transport map is bounded except for a set of small measure. By employing a trace-like inequality on the set where the derivative is bi-Lipschitz, we are able to show the desired distributional result in one dimension, which concludes the one-dimensional version together with the results from \cite{van-Hintum-Spink-Tiba-2}. 

For the higher-dimensional case, we first reduce the matter to the level-set analysis mentioned before. There, by a proposition originally from \cite{Boroczky-Figalli-Ramos}, the one-dimensional estimates automatically yield bounds on the distribution functions in higher dimensions. This guarantees that the log-hypographs in question are close by, in terms of a uniform parameter. At that point, instead of using the result in \cite{Figalli-Jerison-2} (as previously done in \cite{Boroczky-Figalli-Ramos}), we exploit the recent \emph{sharp} stability result by the first author, van Hintum, and Tiba \cite{Figalli-van-Hintum-Tiba}. This enables us to not lose a dimensional constant in the exponents. Since the $\tau$-dependency was only stemming from either the one-dimensional results -- from which we removed such a dependency -- or the Brunn--Minkowski stability problem, this argument yields Theorem \ref{thm:uniform-high-d} as a consequence. 
\end{remark}

The next result that we present in this manuscript deals with a question originally raised by B\"or\"oczky and Ball in \cite{Boroczky-Ball-1}. In analogy to the Brunn-Minkowski theory presented above, where sharp stability for convex sets may be derived from an argument using a Poincar\'e-type inequality, one may wonder what happens if we suppose beforehand that either $h$ or $f$ and $g$ are already \emph{log-concave}. In the one-dimensional case, we are able to prove a \emph{sharp} stability result. 

\begin{theorem}\label{thm:log-conc-sharp} Let $h,f,g:\R\to \R_+$ be measurable functions satisfying \eqref{eq:PL-condition}. Suppose also that either $h$ is log-concave, or both $f$ and $g$ are log-concave, and that  
\begin{equation}\label{eq:PL-almost-eq-1}
\int h \le (1+\varepsilon) \left( \int f \right)^{\lambda} \left( \int g \right)^{1-\lambda}. 
\end{equation}
Then there exist $C,a>0, \, x_0 \in \R$, and $\tilde{h}:\R\to\R_+$ a log-concave function, such that 
\begin{align}\label{eq:almost-eq-conclusion}
\int_{\R} |f(x) - a^{-\lambda} \tilde{h}(x+\lambda x_0)| \, \d x \le C \left( \frac{\varepsilon}{\tau}\right)^{1/2} \int_{\R} f(x) \, \d x, \cr 
\int_{\R} |g(x) -a^{1-\lambda} \tilde{h}(x-(1-\lambda)x_0)| \, \d x \le C \left( \frac{\varepsilon}{\tau}\right)^{1/2} \int_{\R} g(x) \, \d x, \cr 
\int_{\R} |h(x) - \tilde{h}(x)| \, \d x \le C \left( \frac{\varepsilon}{\tau}\right)^{1/2} \int_{\R} h(x) \, \d x,
\end{align}
where $\tau = \min(\lambda,1-\lambda).$ This result is sharp, in the sense that there are functions $h,f,g$ satisfying \eqref{eq:PL-condition} and \eqref{eq:PL-almost-eq} for which \eqref{eq:almost-eq-conclusion} is reversed (with a smaller constant $C$). 
\end{theorem} 

\begin{remark} The proof of this result exploits the interplay between $f$ and $g$ provided by the transport map between them (when looked at as probability measures on the real line). This allows us to exploit a reduction from \cite{Boroczky-Ball-1}, which states that it is sufficient to prove that $f$ and $g$ are close to each other, and then stability follows in the same order. 
In the proof, a key fact is to show that outside an interval where both $f$ and $g$ have negligible mass, the transport map between the densities $f$ and $g$ is universally bi-Lipschitz continuous - which is a property even \emph{stronger} than the one needed for us to prove Theorem \ref{thm:uniform-high-d}.
Then, the conclusion of the proof runs through an inequality that follows directly from the optimal transport proof of the Pr\'ekopa--Leindler inequality. A crucial step is the use of a suitable trace-like inequality for BV functions showing that the distance between $f$ and $g$ is bounded by the quantity $\int_{\R} f(x) |T'(x)-1| \, \d x$, which, by the reductions above, is in turn bounded by $C \varepsilon^{1/2}.$ Thanks to the fact that we are working in dimension 1, our desired trace-like inequality follows from a simple integration by parts argument. 
\end{remark}

A consequence of the methods used to prove Theorem \ref{thm:log-conc-sharp} is a sharp stability result for log-concave \emph{radial} functions in higher dimensions: 

\begin{theorem}\label{thm:radial-sharp} Let $h,f,g:\R^n \to \R_+$ be radial functions satisfying \eqref{eq:PL-condition}. Suppose moreover that either $h$ is log-concave or that $f,g$ are log-concave, and that
\[
\int_{\R^n} h(x) \, \d x \le (1+\varepsilon) \left( \int_{\R^n} f(x) \, \d x \right)^{\lambda} \left( \int_{\R^n} g(x) \, \d x\right)^{1-\lambda}.
\]
Then there is a dimensional absolute constant $C_n > 0,$ a scalar $a>0$ and a radial log-concave function $\tilde{h}$ such that 
\[
\int_{\R^n} |f(x) - a^{\lambda} \tilde{h}(x)| \, \d x \le C_n \left( \frac{\varepsilon}{\tau}\right)^{1/2} \int_{\R^n} f(x) \, \d x, 
\]
\[
\int_{\R^n} |g(x) - a^{-(1-\lambda)} \tilde{h}(x)| \, \d x \le C_n \left( \frac{\varepsilon}{\tau}\right)^{1/2} \int_{\R^n} g(x) \, \d x,
\]
\[
\int_{\R^n} |h(x) - \tilde{h}(x)| \, \d x \le C_n \left( \frac{\varepsilon}{\tau}\right)^{1/2} \int_{\R^n} h(x) \, \d x. 
\]
\end{theorem}

\subsection*{Acknowledgements} We would like to thank K\'aroly B\"or\"oczky for valuable comments on the main results of this manuscript. A.F. is partially supported by the Lagrange Mathematics and Computation Research Center. 

\section{Preliminaries}\label{sec:prelim}

\subsection{Notation} Throughout this manuscript, we will write $\tau = \min(\lambda, 1-\lambda).$ We will generally use the following notation for level sets of functions $f,g,h$ below:
\begin{align}\label{eq:level-sets-def}
A_t &= \{ x \in \R^n \colon f(x) \ge t\}, \cr 
B_t &= \{ x \in \R^n \colon g(x) \ge t\}, \cr 
C_t &= \{ x \in \R^n \colon h(x) \ge t\}.
\end{align}
Given positive quantities $a$ and $b$, we will sometimes write $a \lesssim b,$ meaning that $a \le c \cdot b,$ where $c > 0$ is an absolute constant, depending only possibly on the dimension. We will also write $c(\tau)$ to denote an absolute computable function of only $\tau,$ that may change from line to line. 

\subsection{Transport maps and trace-like inequalities} In the proof of the main result of this manuscript, a crucial tool is the use of the transport map between radially decreasing functions. We first state the main properties we are going to need about that object. 

\begin{proposition}\label{prop:f-g-transp} Let $f,g$ be two non-negative, radially decreasing probability distributions. Let $T$ denote the transport map between $f$ and $g$, in the sense that
	\begin{equation}\label{eq:transport-f-g}
	\int_{-\infty}^{x} f(t) \, dt = \int_{-\infty}^{T(x)} g(t) \, dt, \qquad \forall \, x \in \R. 
	\end{equation}
Then the map $T$ is an increasing bijection of the real line, and hence differentiable almost everywhere. Furthermore, it satisfies the equation 
    \[
    f(x) = g(T(x))\cdot T'(x),\quad  \text{ for almost every } x \in \R. 
    \]
    Moreover, if we denote the inverse of $T$ by $S,$ we have that 
    \[
    g(y) = f(S(y))\cdot S'(y),\quad  \text{ for almost every } y \in \R. 
    \]
\end{proposition}

\begin{proof} We simply note that \eqref{eq:transport-f-g} is equivalent to the fact that the map $T$ is a \emph{transport map} between the measures $f(x) \, \d x$ and $g(y) \, \d y.$ The properties of $T$ then follow, for instance, from the theory of optimal transport as in \cite{Figalli-Glaudo}, restricted to the one-dimensional case. 
\end{proof}

We now state a trace-like inequality result for radially decreasing functions on the real line, which is the main new bridge in order to prove the crucial step in the one-dimensional version of Theorem \ref{thm:uniform-high-d}, as well as in the sharp results Theorems \ref{thm:log-conc-sharp} and \ref{thm:radial-sharp}, which is Proposition \ref{prop:f-g-dist} below.  

\begin{proposition}\label{prop:trace-log-conc} Suppose that $f:\R \to \R_+$ is a $L^1$ function which is furthermore increasing on $(-\infty,0)$ and decreasing on $(0,+\infty).$ Moreover, let $\Phi$ be a locally Lipschitz function with $\Phi(0) = 0.$ Under those assumptions, the inequality
\begin{equation}\label{eq:trace} 
	\int_{\R} |\Phi(x)| \, |\d f(x)| \le  \int_{\R} |f(x)||\Phi'(x)| \, \d x
	\end{equation}
holds, where $|\d f|$ denotes the variation of the measure $\d f$ such that $f(x) = \int_{-\infty}^x \, \d f(s)$.  
\end{proposition}
	
\begin{proof}	We note that, since $f$ is increasing on $(-\infty,0)$ and decreasing on $(0,+\infty),$  we may write the left-hand side of \eqref{eq:trace} as 
	\[
	\int_{-\infty}^0  |\Phi(x)| \, \d f(x) - \int_0^{+\infty} |\Phi(x)| \, \d f(x). 
	\]
    Suppose first $f$ is compactly supported, so that both integrals exist and are finite. Then we may use a Riemann--Stieltjes integration by parts together with the fact that $\Phi(0) = 0$ and that the support of $f$ is compact: 
    \begin{align*}
        \int_{-\infty}^0 |\Phi(x)| \, \d f(x) &= - \int_{-\infty}^0 f(x)| \frac{\d}{\d x} |\Phi|(x) \, \d x, \cr  \int_0^{\infty} |\Phi(x)| \, \d f(x) &= - \int_0^{\infty} f(x) \frac{\d}{\d x} |\Phi|(x) \, \d x.
    \end{align*}
    We readily see that the difference between the right-hand sides above is bounded by
	\[
	\int_{-\infty}^{\infty}  |f(x)|\left|  \frac{\d}{\d x} |\Phi(x)| \right|\, \d x,
	\]
	which, on the other hand, is bounded by the asserted quantity. For the non-compactly supported case, we simply argue by approximation in a standard way. 
 \end{proof} 

\subsection{The Pr\'ekopa--Leindler inequality and rearrangements} Our next preliminary result concerns the interplay of rearrangements with the Pr\'ekopa--Leindler inequality. We define, for a function $\varphi:\R^n \to \R_+,$ its \emph{symmetric decreasing rearrangement} to be the unique radial function $\varphi^*:\R^n \to \R_+$ such that 
\[
|\{\varphi > s\}| = |\{ \varphi^* > s\}|,
\]
for all $s > 0$ for which the function $\mu(s) = |\{ \varphi > s\}|$ is continuous. A crucial property of the rearrangement is that it preserves functions satisfying \eqref{eq:PL-condition}: 

\begin{proposition}\label{lemma:rearrangements} Let $h,f,g:\R^n \to \R$ satisfy  \eqref{eq:PL-condition}. Then the same inequality follows for their rearrangements. That is, 
\[
h^*\left( \lambda x + (1-\lambda)y \right) \ge f^*(x)^{\lambda} g^*(y)^{1-\lambda}, \qquad \forall \, x,y \in \R^n.
\]
\end{proposition}
\begin{proof} Let $H, F,G$ denote, respectively, the distribution functions of $h,f,g.$ That is, 
\begin{align*}
H(t) &= \mathcal{H}^n(C_t), \cr 
F(t) &= \mathcal{H}^n(A_t), \cr
G(t) &= \mathcal{H}^n(B_t).
\end{align*}
By \eqref{eq:PL-condition}, we have that 
\[
C_{s^{\lambda}t^{1-\lambda}} \supset \lambda A_s + (1-\lambda)B_t,
\]
for all $s,t>0.$ Hence, by Brunn--Minkowski, it follows that 
\begin{equation}\label{eq:ineq-mu}
H(s^{\lambda} t^{1-\lambda}) \ge \left(\lambda F(s)^{1/n} + (1-\lambda)G(t)^{1/n}\right)^{n} \qquad \forall \, s,t > 0.
\end{equation}
We then use that, for each $x \in \R^n,$ the rearrangement of a function $\varphi$ may be written as 
\[
\varphi^*(x) = \sup \{ t > 0 \colon \Phi(t) \ge \vol(B_{|x|}(0))\},
\]
where $\Phi$ is the distribution function of $\varphi.$ Take then $s_1$ for which $F(s_1) \ge \vol(B_{|x|}(0))$ and $t_1$ for which $G(t_1) \ge \vol(B_{|y|}(0))$ in \eqref{eq:ineq-mu}. Then it follows that 
\[
H(s_1^{\lambda} t_1^{1-\lambda}) \ge \vol(B_{\lambda |x|+(1-\lambda)|y|}(0)). 
\]
Since $s_1$ and $t_1$ can be made arbitrarily close to $f^*(|x|)$ and $ g^*(|y|),$ respectively, we conclude that 
$$f^*(|x|)^{\lambda} g^*(|y|)^{1-\lambda} \le h^*\left(\lambda |x| + (1-\lambda)|y|\right), \qquad \forall \, x,y \in \R^n.$$
Since $h^*$ is radially decreasing, this implies the claim.
\end{proof}

\subsection{Log-concave functions} Finally, we will need some properties of log-concave functions on the real line for Theorems \ref{thm:log-conc-sharp} and \ref{thm:radial-sharp}. We say that a function $\varphi:\R^n \to \R_+$ is \emph{log-concave} if, for any $x,y \in \R^n$ and $\lambda \in (0,1),$ we have 
\[
\varphi(\lambda x + (1-\lambda)y) \ge \varphi(x)^{\lambda} \varphi(y)^{1-\lambda}. 
\]
We will focus on properties of log-concave functions in dimension $n=1.$ For that case, if we have that $\int_{\R} \varphi(x) \, \d x = 1,$ then we say that $\varphi$ has \emph{median} $m_{\varphi}$ if
\[
\int_{-\infty}^{m_{\varphi}} \varphi(x) \, \d x = \int_{m_{\varphi}}^{\infty} \varphi(x) \, \d x = \frac{1}{2}.
\]
It turns out that the median plays a special role for pointwise estimates for log-concave probability distributions, as highlighted by the next result, which is a particular case of Proposition~2.2 in \cite{Boroczky-Ball-2}. 

\begin{proposition}\label{prop:first-log-conc} Let $\varphi:\R \to \R_+$ be a log-concave probability distribution with median $m$. We have then that 
\begin{equation}\label{eq:comparis-log-conc}
\varphi(m) e^{-2\varphi(m)|x-m|} \le \varphi(x) \le \varphi(m) e^{2\varphi(m)|x-m|},
 \end{equation}
  whenever $|x-m| \le \frac{\log(2)}{2\varphi(m)}$.
\end{proposition}
\begin{proof} By scaling and translating, we may suppose, without loss of generality, that $m=0$ and $\varphi(0) = \frac{1}{2}.$ We then let $\psi:\R \to \R_+$ be the log-concave distribution defined by
$$\psi(x) = \left\{\begin{array}{ll}
\frac{1}{2} e^{-x} &\text{whenever $x \ge - \log(2),$}\\
0&\text{otherwise.}
\end{array}
\right.
$$
Then $\psi(0) = \varphi(0)$ and $m_{\psi} = 0$ as well. Hence, since $\int_{0}^{\infty} \varphi(x) \, \d x = \int_{0}^{\infty} \psi(x) \, \d x,$ there must exist $v > 0$ for which $\varphi(x) \le \psi(x)$ for all $x \ge v.$ Take the minimal $v.$ It then follows by log-concavity of $\varphi$ that $\varphi(x) \ge \psi(x)$ for $x \in [0,v],$ and $\varphi(x) \le \psi(x)$ otherwise. This implies in particular the claimed upper bound in \eqref{eq:comparis-log-conc}. 

In order to prove the lower bound, we suppose without loss of generality that $x>0$ and note that it is enough to prove that 
$\varphi(\log(2)) \ge \frac{1}{4},$
since the desired assertion for points in the interval $[0,\log(2)]$ follows by log-concavity directly. Suppose then, for the sake of a contradiction, that it does not hold - that is, $\varphi(\log(2)) < \frac14$. Then we should have, by log-concavity of $\varphi$ and the fact that $\varphi(0) = \frac12,$ that 
\[
\varphi(x) \le \frac{1}{4} e^{-a(x-\log(2))},
\]
for some $a \ge 1.$ We fix then $t_0 = \frac{a-1}{a+1} \log(2),$ and note that, since $\frac{1}{4} e^{-a(t_0-\log(2))} = \frac{1}{2} e^{-t_0},$ then again by log-concavity we get $\varphi(x) \le \frac12 e^{x}$ for $x \in [0,t_0].$ We then estimate 
\[
\int_0^{\infty} \varphi(x) \, \d x < \frac12 \int_0^{t_0} e^{x} \, \d x + \frac14 \int_{t_0}^{\infty} e^{-a(t-\log(2))} \, \d t < \frac12,
\]
as $a \ge 1.$ This is a contradiction to the fact that $0$ is the meadian of $\varphi,$ which implies the claim.
\end{proof}

The next result shows how to obtain a generalization of the result in the previous Proposition where one does not compare the log-concave distribution to the median point. It is a combination of results from Corollary~2.3 and Proposition~2.2 in \cite{Boroczky-Ball-2}. 

\begin{proposition}\label{prop:log-conc}
Let $\varphi:\R \to \R_+$ be a log-concave probability density, and let $\int_{x}^{\infty} \varphi=\nu \in\left(0, \frac{1}{2}\right]$. Then we have the following: 

\begin{enumerate}
    \item[(i)] $\varphi(x) \cdot e^{-\frac{\varphi(x)|t-x|}{\nu}} \leq \varphi(t) \leq \varphi(x) \cdot e^{\frac{\varphi(x)|t-x|}{\nu}} \quad$ if $|t-x| \leq \frac{\nu \log(2)}{\varphi(x)}$;
    \item[(ii)] $\varphi(w) \le \varphi (x),$ for all $w > x.$ 
\end{enumerate}
\end{proposition}

\begin{proof} 

\vspace{1mm}

We begin by proving (i). Let $|t-x| \leq \frac{\nu \log(2)}{\varphi(x)}$. There exists a unique $\lambda \in \mathbb{R}$ such that, for the function
$$
\tilde{\varphi}(t)=\left\{\begin{array}{ll}
\varphi(t) & \text { if } t \geq x \\
\min \left\{\varphi(t), \varphi(x) \cdot e^{\lambda(t-x)}\right\} & \text { if } t \leq x 
\end{array},\right.
$$
we have $\int_{-\infty}^{x} \tilde{\varphi}=\nu$. We note that $\tilde{\varphi}$ is log-concave and $\lambda \geq \frac{-\varphi(x)}{\nu}$. In particular $\frac{1}{2 \nu} \tilde{\varphi}$ is a log-concave probability distribution whose median is $x$, and hence Proposition \ref{prop:first-log-conc} yields $\varphi(t) \geq \tilde{\varphi}(t) \geq \varphi(x) \cdot e^{\frac{-\varphi(x)|t-x|}{\nu}}$. Since for $s=2 x-t$ we have $\varphi(s) \geq \varphi(x) \cdot e^{\frac{-\varphi(x)|s-x|}{\nu}}$, we conclude (i) by log-concavity. 

For the proof of (ii), we may translate and dilate (preserving $L^1$ norms) in order to assume without generality that $x=0$ and $\varphi(0) = 1/2.$ From this, the assertion in part (i) above implies that we may suppose $w > 2 \nu \log(2).$ 

Suppose then that $\varphi(w) > 1.$ Then, log-concavity of $\varphi$ directly implies that $\varphi(t) \ge \frac{1}{2} e^{(t/w) \log(2\varphi(w))},$ for each $t \in (0,w).$ We have then 
\[
\nu \ge \int_0^w \varphi(t) \, dt \ge \frac{w(\varphi(w) - 1)}{2 \log(2 \varphi(w))} \ge \nu \log(2) \frac{2\varphi(w) - 1}{ \log(2\varphi(w))}.
\]
It is then a simple computation to verify that the function $s \mapsto \frac{s-1}{\log(s)}$ is larger than $\frac{1}{\log(2)}$ for $s > 2.$ Hence, the right-hand side above is bounded from below \emph{strictly} by $\nu$ under our hypothesis, a contradiction stemming from $\varphi(w) > 1.$ Hence $\varphi(w) \le 1,$ and the second assertion is proved. 
\end{proof}

We finally note the following simple result for log-concave functions, which will play a crucial role in extending the results from the $\lambda = 1/2$ case to the case of \emph{general} $\lambda \in (0,1)$:

\begin{proposition}[Lemma~7.4 in \cite{Boroczky-De}]\label{prop:interpolation-log-concave} For fixed $\lambda \in(0,1)$, if $\eta \in(0,2 \cdot \min \{1-\lambda, \lambda\})$ and $\varphi$ is a log-concave function on $[0,1]$ satisfying $\varphi(\lambda) \leq(1+\eta) \varphi(0)^{1-\lambda} \varphi(1)^\lambda$, then
$$
\varphi\left(\frac{1}{2}\right) \leq\left(1+\frac{\eta}{\min \{1-\lambda, \lambda\}}\right) \sqrt{\varphi(0) \varphi(1)}
$$
\end{proposition}

\begin{proof} We may assume that $0<\lambda<\frac{1}{2}$. Then, since $\lambda=(1-2 \lambda) \cdot 0+2 \lambda \cdot \frac{1}{2}$ and $\varphi(\lambda) \leq(1+\eta) \varphi(0)^{1-\lambda} \varphi(1)^\lambda$, the log-concavity of $\varphi$ yield
$$
(1+\eta) \varphi(0)^{1-\lambda} \varphi(1)^\lambda \geq \varphi(\lambda) \geq \varphi(0)^{1-2 \lambda} \varphi\left(\frac{1}{2}\right)^{2 \lambda}.
$$
Thus $(1+\eta)^{\frac{1}{2 \lambda}} \leq e^{\frac{\eta}{2 \lambda}} \leq 1+\frac{\eta}{\lambda}$ implies
$$
\varphi\left(\frac{1}{2}\right) \leq(1+\eta)^{\frac{1}{2 \lambda}} \sqrt{\varphi(0) \varphi(1)} \leq\left(1+\frac{\eta}{\lambda}\right) \sqrt{\varphi(0) \varphi(1)}.
$$
\end{proof} 
 
\section{Proof of Theorem \ref{thm:uniform-high-d}}

We will divide our discussion of the proof of Theorem \ref{thm:uniform-high-d} into two parts: the one-dimensional and the higher-dimensional parts. For each part, the proof will be split into several steps. 

\vspace{3mm}

\subsection{Part I: One-dimensional analysis} In this part, we show a new alternative approach to the stability of the Pr\'ekopa--Leindler inequality in dimension $n=1.$ This new approach has the major advantage of removing the exponent dependency on $\lambda$ in that case. In fact, in the $n=1$ case, the result of Theorem \ref{thm:uniform-high-d} can be deduced from \cite{van-Hintum-Spink-Tiba-2} directly, since it uses only the two-dimensional sharp Brunn--Minkowski inequality as a black box. 

\vspace{2mm}

\noindent\textbf{Step 1. Bounding the level sets of near-extremals.} We first discuss some more immediate properties of functions  $h,f,g:\R\to\R_+$  satisfying  \eqref{eq:PL-condition} and \eqref{eq:PL-almost-eq} for $n=1.$

First of all, up to multiplying $f,g,h$ by constants that preserve the relations between them, and up to a rescaling in the variable $x$, one can always assume that 
 \begin{equation}
 \label{eq:normalize fg}\|f\|_\infty =1\qquad\text{and}\qquad \|f\|_{1} = \|g\|_{1} = 1.\end{equation}
Then, it follows by \cite[Lemma~2.4]{Boroczky-Figalli-Ramos} that 
\begin{equation}\label{eq:L^1-norm-bound} \|g\|_{\infty} \in (1-c(\tau) \varepsilon^{1/2},1+c(\tau) \varepsilon^{1/2})
\end{equation} 
whenever $\varepsilon \lesssim \tau^3.$

Given the above reductions, we recall the following result, which first appeared in \cite[Lemma~2.5]{Boroczky-Figalli-Ramos}. We refer the reader to that manuscript for a proof of this result. 

\begin{lemma}
\label{thm:cutting-support} 
Let $h,f,g$ satisfy \eqref{eq:PL-condition}, \eqref{eq:PL-almost-eq}, and \eqref{eq:normalize fg}, and assume that $\varepsilon \lesssim \tau^3$. If $\varepsilon^{1/2} \leq \eta<1$, then  
\begin{equation}
\label{cutting-supporteq1}
\mathcal{H}^1(\{f\geq \eta\})\lesssim  \tau^{-\frac52} \cdot|\log \varepsilon|^{\frac4\tau} ,\qquad 
\mathcal{H}^1(\{g \geq \eta\}) \lesssim  \tau^{-\frac52} \cdot|\log \varepsilon|^{\frac4\tau},
\end{equation}
and 
\[
\int_{\{f < \eta\}} f
\lesssim \tau^{-\frac52} \cdot\eta\,|\log \varepsilon|^{\frac4\tau},\qquad \int_{\{g <\eta\}} g \lesssim \tau^{-\frac52} \cdot\eta\,|\log \varepsilon|^{\frac4\tau}.
\]
\end{lemma}

\vspace{2mm}

\noindent\textbf{Step 2. Reduction to the radial case.} Let $h,f,g$ satisfy \eqref{eq:PL-condition}, \eqref{eq:PL-almost-eq}, and \eqref{eq:normalize fg}. Given $\theta>0$ a small constant to be fixed later, define the truncated log-hypographs of $f,g,h$ as 
\begin{equation*}\begin{split}
\mathcal{S}_f = \{ (x,T) \in \R^{2} \colon x \in \{ f > \eps^{\theta}\}, \, \eps^{\theta} \le e^T < f(x)\},\cr
\mathcal{S}_g = \{ (x,T) \in \R^{2} \colon x \in \{ g > \eps^{\theta}\}, \, \eps^{\theta} \le e^T < g(x)\},\cr
\mathcal{S}_h = \{ (x,T) \in \R^{2} \colon x \in \{ h > \eps^{\theta}\}, \, \eps^{\theta} \le e^T < h(x)\}. 
\end{split}\end{equation*}
We will start by controlling the measure of $\mathcal{S}_f$ and $\mathcal{S}_g.$ 

We have directly by the definitions that 
\begin{equation}\label{eq:bound-level-f}
\frac{1}{2} \le \mathcal{H}^2(\mathcal{S}_f) \le \theta |\log \varepsilon| \mathcal{H}^1(A_{\varepsilon^{\theta}}) \lesssim c(\tau) |\log \varepsilon|^{\frac{4}{\tau} + 1}. 
\end{equation}
The same argument can be used in order to prove entirely analogous bounds for $\mathcal{H}^2(\mathcal{S}_g).$ Now, the condition \eqref{eq:PL-condition} implies that 
\[
\lambda \mathcal{S}_f + (1-\lambda) \mathcal{S}_g \subset \mathcal{S}_h.
\]
The goal is then to show that the triple $(\mathcal{S}_h,\mathcal{S}_f,\mathcal{S}_g)$ satisfies near-equality in the Brunn--Minkowski inequality, with deficit given by an absolute power of $\varepsilon.$ We shall use the following numerical lemma: 

\begin{lemma}\label{lemma:numerical} Let $x,y,z > 0$ be such that $x \ge \left( \lambda y^{1/2} + (1-\lambda) z^{1/2}\right)^2.$ Then 
\[
0 \le x - \left( \lambda y^{1/2} + (1-\lambda) z^{1/2}\right)^2 \le \left| x - \left( \lambda y + (1-\lambda) z\right) \right| + \tau |y^{1/2}-z^{1/2}|^2.
\]
\end{lemma}

\begin{proof} It suffices to verify that 
\[
\left( \lambda y^{1/2} + (1-\lambda) z^{1/2}\right)^2 \ge \lambda y + (1-\lambda)z - \tau|z^{1/2}-y^{1/2}|^2.
\]
This assertion is equivalent to 
\[
\tau |z^{1/2} - y^{1/2}|^2 \ge \lambda(1-\lambda)|z^{1/2} - y^{1/2}|^2,
\]
which is true by the definition of $\tau$. 
\end{proof}

Using Lemma \ref{lemma:numerical} with $x = \mathcal{H}^2(\mathcal{S}_h), y = \mathcal{H}^2(\mathcal{S}_f), z = \mathcal{H}^2(\mathcal{S}_g),$ we readily obtain that 
\begin{align}\label{eq:-B-M-attempt}
0 & \le \mathcal{H}^2(\mathcal{S}_h) - \left(\lambda \mathcal{H}^2(\mathcal{S}_f)^{1/2} + (1-\lambda) \mathcal{H}^2(\mathcal{S}_g)^{1/2} \right)^2 \cr 
 & \le \int_{\varepsilon^{\theta}}^2 \left| \mathcal{H}^1(C_t) - \lambda \mathcal{H}^1(A_t) - (1-\lambda) \mathcal{H}^1(B_t) \right| \, 
 \frac{\d t}{t} + \tau \left| \mathcal{H}^2(\mathcal{S}_f)^{1/2} - \mathcal{H}^2(\mathcal{S}_g)^{1/2}\right|^2 \cr 
 & \le \varepsilon^{-\theta} \int_{0}^2 \left| \mathcal{H}^1(C_t) - \lambda \mathcal{H}^1(A_t) - (1-\lambda) \mathcal{H}^1(B_t) \right| \, 
 \d t \cr 
 & + \frac{\tau}{\varepsilon^{2\theta} \max\{\mathcal{H}^2(\mathcal{S}_f), \mathcal{H}^2(\mathcal{S}_g)\}} \left(\int_0^2 |\mathcal{H}^1(A_t) - \mathcal{H}^1(B_t)| \, \d t\right)^2.
\end{align}
Now, by invoking \cite[Lemma~3.2]{Boroczky-Figalli-Ramos}, we have that 
\[
\int_0^2 \left| \mathcal{H}^1(C_t) - \lambda \mathcal{H}^1(A_t) - (1-\lambda) \mathcal{H}^1(B_t) \right| \, 
 \d t \lesssim \tau^{-\frac{3}{2}} \varepsilon^{1/2}. 
\]
Now, assume that  we can bound 
\begin{equation}\label{eq:bound-distribution-f-g}
\int_0^2 |\mathcal{H}^1(A_t) - \mathcal{H}^1(B_t)| \, \d t \lesssim \varepsilon^{\theta_0},
\end{equation}
for some absolute constant $\theta_0 > 0.$ 
Then, choosing $\theta >0$  to be a small enough absolute constant and recalling \eqref{eq:bound-level-f} (and its analogue for $g$), we conclude that
\begin{equation}
\label{eq:S fgh}
     \mathcal{H}^2(\mathcal{S}_h) \le (1+c(\tau) \varepsilon^{\alpha_0}) \left(\lambda \mathcal{H}^2(\mathcal{S}_f)^{1/2} + (1-\lambda) \mathcal{H}^2(\mathcal{S}_g)^{1/2} \right)^2
\end{equation}
for some $\alpha_0>0$ universal.

Notice now that the left-hand side of \eqref{eq:bound-distribution-f-g} equals 
\[
\int_{\R} |f^*(x) - g^*(x)| \, \d x,
\]
where $\varphi^*$ denotes the radially decreasing rearrangement of a given function $\varphi$, as defined in Section \ref{sec:prelim}. Also, using Proposition \ref{lemma:rearrangements}, it follows that $h^*,f^*,g^*$ also satisfy \eqref{eq:PL-condition} and \eqref{eq:PL-almost-eq}.

In other words, we have reduced the validity of \eqref{eq:S fgh} to proving a uniform closeness relationship between functions $f,g$ which are \emph{even} and \emph{radially decreasing}, in addition to satisfying \eqref{eq:PL-condition}, \eqref{eq:PL-almost-eq}, and \eqref{eq:normalize fg}

\vspace{2mm}

\noindent\textbf{Step 3. Uniform bound for the even case.} Next, we show the following result, which confirms the validity of a uniform estimate in the even case: 

\begin{proposition}\label{prop:f-g-dist} Suppose $h,f,g:\R \to \R_+$ satisfy \eqref{eq:PL-condition} and \eqref{eq:PL-almost-eq}. Suppose, moreover, that $f$ and $g$ are both even and radially decreasing, $\int_{\R} f = \int_{\R} g = 1$, and $\max(\|f\|_{\infty},\|g\|_{\infty}) \le 2$. Then, for any $\gamma > 0, \tau \in (0,1)$, we may find $c_{\gamma}(\tau) > 0$ such that 
\begin{equation}\label{eq:closeness-f-g}
\int_{\R} |f(x) - g(x)| \, \d x \lesssim c_{\gamma} (\tau) \varepsilon^{1/2 - \gamma}. 
\end{equation}
\end{proposition}

\begin{proof} Let $I_1 = \{ f > \varepsilon^{1/2}\}$ and $I_2 = \{ g > \varepsilon^{1/2}\}$. Since $f$ and $g$ are both even and decreasing, either $I_1 \supset I_2$ or $I_2 \supset I_1$. Suppose without loss of generality that $I_1 \supset I_2$ so that $\{f \le \varepsilon^{1/2}\} \subset \{ g \le \varepsilon^{1/2}\}$. 

Consider then transport map between $f$ and $g,$ as in Proposition \ref{prop:f-g-transp}. Since $f$ and $g$ are origin-symmetric, we must have $T(0) = 0$. We then split the integral on the left-hand side of \eqref{eq:closeness-f-g} into two parts: 
\[
\int_{\R} |f(x) - g(x)| \, \d x = \int_{\{f > \varepsilon^{1/2}\}} |f(x) - g(x)| \, \d x + \int_{\{ f \le \varepsilon^{1/2}\}} |f(x) - g(x)| \, \d x =: I_1 + I_2. 
\]
For $I_2$, we bound it by 
\begin{align*}
I_2 \le \int_{\{f \le \varepsilon^{1/2}\}} f + \int_{\{f \le \varepsilon^{1/2}\}} g 
  \le \int_{\{f \le \varepsilon^{1/2}\}} f + \int_{\{g \le \varepsilon^{1/2}\}} g,
 \end{align*}
 by the reductions made. By Lemma \ref{thm:cutting-support}, both integrals above are bounded by $c_{\gamma}(\tau) \varepsilon^{1/2 - \gamma}$, and hence we focus on $I_1.$ There, we divide into two further parts, using the transport map between $f$ and $g$. Indeed, let
 $$S_0 = \{ x \in \R \colon \text{ either } T'(x) > 10 \text{ or } T'(x) < 1/10\}.$$
 We start with the following proof of the Pr\'ekopa--Leindler inequality: suppose $\int_{\R} f(x) \, \d x = \int_{\R} g(x) \, \d x = 1.$ Then 
\begin{align*}
1 & = \int_{\R} f(x) \, \d x = \int_{\R} f(x)^{\lambda} g(T(x))^{1-\lambda} (T'(x))^{1-\lambda} \, \d x \cr 
  & \le \int_{\R} h\left( \lambda x+(1-\lambda)T(x)\right) (T'(x))^{1-\lambda} \, \d x \cr 
  & \le \int_{\R} h\left( \lambda x+(1-\lambda)T(x)\right) \left(\lambda + (1-\lambda) T'(x)\right) \, \d x = \int_{\R} h(x) \, \d x.
\end{align*}
If, however, we know that $\int h \le 1+\varepsilon,$ then a thorough analysis of the proof above yields 
\begin{align}\label{eq:deficit-estimate}
\varepsilon & \ge \int_{\R} h(\lambda x + (1-\lambda)T(x)) \left( (\lambda + (1-\lambda)T'(x)) - (T'(x))^{1-\lambda}\right) \, \d x \cr 
            & \ge \int_{\R} f(x) \left( \frac{\lambda + (1-\lambda)T'(x)}{(T'(x))^{1-\lambda}} - 1 \right) \, \d x \ge \tau \int_{\R} f(x) \frac{(1-\sqrt{T'(x)})^2}{(T'(x))^{1-\lambda}}\, \d x. 
\end{align}
Consider then the splitting 
\[
I_1 = \int_{\{f > \varepsilon^{1/2}\} \cap S_0} |f-g| + \int_{\{f > \varepsilon^{1/2}\} \cap S_0^c} |f-g| =: J_1 + J_2. 
\]
We bound $J_1$ by using the optimal transport approach for Pr\'ekopa--Leindler above: indeed, 
\[
J_1 \le \left(\|f\|_{\infty} + \|g\|_{\infty}\right) |\{f > \varepsilon^{1/2}\} \cap S_0|.
\]
On the other hand, by \eqref{eq:deficit-estimate}, we have 
\begin{align} 
\label{eq:f S0}
\varepsilon  \ge \tau \int_{\{f > \varepsilon^{1/2}\} \cap S_0} f(x) \frac{(1-\sqrt{T'(x)})^2}{(T'(x))^{1-\lambda}}\, \d x 
     \ge c(\tau)^{-1} \cdot \varepsilon^{1/2} |\{ f > \varepsilon^{1/2}\} \cap S_0|, 
\end{align} 
which implies that 
\[
J_1 \le c(\tau) \varepsilon^{1/2}. 
\]
Thus, we can focus on $J_2$. There, we need to use the transport map approach. Write, for shortness, $R_0 = \{f > \varepsilon^{1/2}\} \cap S_0^c$. Then we have 
\begin{align*}
\int_{R_0} |f(x) - g(x)| \, \d x & = \int_{R_0} |g(T(x))T'(x) - g(x)| \, \d x \cr 
    & \le \int_{R_0} |g(T(x))||T'(x) - 1| \, \d x + \int_{R_0} |g(x) - g(T(x))| \, \d x \cr 
    & \le \int_{T(R_0)} |g(y)| |S'(y) - 1| \, \d y + 
    \int_{R_0} |g(x) - g(T(x))| \, \d x \cr 
    & =: K_1 + K_2. 
\end{align*}
Let us first look at what $T(R_0)$ looks like. Since $R_0 \subset S_0^c$, we have $T(R_0) \subset T(S_0^c) = \{T(x) \colon 1/10 < T'(x) < 10\} = \{ y \in \R \colon 1/10 < S'(y) < 10\}.$ We then notice that the estimate \eqref{eq:deficit-estimate} can be done in the exact same fashion with $f$ and $T$ replaced by $g$ and $S$, where one obtains instead that 
\begin{equation}\label{eq:deficit-estimate-2}
\varepsilon \ge \tau \int_{\R} g(y) \frac{(1-\sqrt{S'(y)})^2}{(S'(y))^{\lambda}} \, \d y.  
\end{equation}
Let us use this in order to estimate $K_1$: we have, by Cauchy-Schwarz and the inclusions above, 
\begin{align*}
\int_{T(R_0)} g(y) |S'(y) - 1| \, \d y &\le \left(\int_{\{1/10 < S' < 10\}} g(y) |S'(y)-1|^2 \, \d y \right)^{1/2} \cr 
    & \le c(\tau) \left( \int_{\R} g(y) \frac{(1-\sqrt{S'(y)})^2}{(S'(y))^{\lambda}} \, \d y\right)^{1/2},
 \end{align*}
and, since the right-hand side above is bounded by $c(\tau) \varepsilon^{1/2}$, this bounds $K_1$.

It remains to bound $K_2$. We shall use the fundamental theorem of calculus in order to do that: indeed, first we analyse it inside the set for which $T(x) > x.$ There, we have 
\begin{align*}
\int_{R_0 \cap \{T(x) > x\}} |g(x) - g(T(x))| \, \d x & \le \int_{\R} \left( \int_{\R} \chi_{\{T(x) > s > x\}} \cdot \chi_{R_0}(x) \, \d x \right) \, |\d g(s)| \cr 
    & = \int_{\R} \left( \int_{S(s)}^s \chi_{R_0}(x) \, \d x \right) \, |\d g(s)|.
 \end{align*}
Let then $\Phi(s) = \int_{\R} \chi_{\{ s > x > S(s)\}} \cdot \chi_{R_0} \, \d x.$ This is a locally Lipschitz function satisfying $\Phi(0) = 0.$ Hence, we can apply Proposition \ref{prop:trace-log-conc}. Inequality \eqref{eq:trace} implies that the right-hand side above is bounded by 
 \begin{align*}
 &\int_{\R}  |g(s)| | \chi_{R_0}(s) - S'(s) \chi_{R_0}(S(s))| \, \d s \cr 
 &\le \int_{\R}  \chi_{R_0}(S(s)) |g(s)| \, |S'(s) - 1| \, \d s 
 + \int_{\R} |g(s)| |\chi_{R_0}(s) - \chi_{R_0}(S(s))| \, \d s \cr 
 &\le  \int_{\{1/10 < S' < 10\}} |g(s)| \, |S'(s) - 1| \, \d s + \int_{\R} |g(s)| |\chi_{R_0} - \chi_{T(R_0)}| \, \d s. 
 \end{align*}
 Since the first term in the last line is already bounded by $c(\tau) \varepsilon^{1/2},$ we focus on the second one. This second term is clearly bounded by 
 \[
 \int_{R_0 \triangle T(R_0)} g(s) \, \d s \le \int_{R_0 \setminus T(R_0)} g(s) \, \d s + \int_{T(R_0) \setminus R_0} g(s) \, \d s. 
 \]
    Note that $T(R_0) \setminus R_0 \subset R_0^c \subset \{f < \varepsilon^{1/2}\} \cup S_0$, which we write as the disjoint union $\{ f < \varepsilon^{1/2}\} \sqcup \left(S_0 \cap \{ f > \varepsilon^{1/2}\}\right)$. Hence, 
\begin{align*} 
\int_{T(R_0) \setminus R_0} g(s) \, \d s & \le \int_{\{f \le \varepsilon^{1/2}\}} g(s) \, \d s + \int_{S_0 \cap \{f > \varepsilon^{1/2}\}} g(s) \, \d s \cr 
    & \le \int_{\{ g \le \varepsilon^{1/2}\}} g(s) \, \d s + \|g\|_{\infty} | S_0 \cap \{f > \varepsilon^{1/2}\}| \lesssim c_{\gamma} (\tau) \varepsilon^{1/2 - \gamma},
\end{align*} 
where the last inequality follows from Lemma \ref{thm:cutting-support} and \eqref{eq:f S0}.
Similarly, we have 
\begin{align*} 
R_0 \setminus T(R_0) & \subset T(R_0^c) \subset \left( T(\{f < \varepsilon^{1/2}\}) \cup T(S_0)\right)  \cr 
& = \left( T(\{ f < \varepsilon^{1/2}\}) \sqcup (T(S_0) \cap T(\{f \ge \varepsilon^{1/2}\})) \right) 
\end{align*}
Note that $T(S_0) = \{ y=  T(x) \colon \text{ either } 1/10 > T'(x) \text{ or } T'(x) > 10\} = \{ y \in \R \colon \text{ either } S'(y) < 1/10 \text{ or } S'(y) > 10\}.$ From restricting the integral on the right-hand side of \eqref{eq:deficit-estimate-2} to $T(S_0)$, we get that 
\[
c(\tau) \varepsilon \ge \int_{T(S_0)} g(s) \, \d s. 
\]
Thus, using again Lemma \ref{thm:cutting-support}, we have
\begin{align*}
\int_{R_0 \setminus T(R_0)} g(s) \, \d s & \le \int_{T(\{f < \varepsilon^{1/2}\})} g(s) \, \d s + \int_{T(S_0)} g(s) \, \d s \cr 
    & \le \int_{\{f < \varepsilon^{1/2}\}} f(s) \, \d s + c(\tau) \varepsilon \le c_{\gamma} (\tau) \varepsilon^{1/2 - \gamma}. 
\end{align*}
 By repeating the exact same analysis as above in the set $\{T(x) < x\}$, this allows us to conclude the desired estimate on the distance between $f$ and $g,$ concluding our proof.
\end{proof}

\vspace{2mm}

\noindent\textbf{Step 4. Conclusion.} We are now ready to prove the $n=1$ case of Theorem \ref{thm:uniform-high-d}.  

Indeed, applying Proposition \ref{prop:f-g-dist}
to $f^*,g^*,h^*$, we get
\[
\int_{\R} |f^*(x) - g^*(x)| \, \d x \le c_{\gamma} (\tau) \varepsilon^{1/2 - \gamma}.
\]
Then, combining this bound with \eqref{eq:-B-M-attempt}, for $\theta$ sufficientyl small we obtain 
\[
\mathcal{H}^2(\mathcal{S}_h) - \left(\lambda \mathcal{H}^2(\mathcal{S}_f)^{1/2} + (1-\lambda) \mathcal{H}^2(\mathcal{S}_g)^{1/2} \right)^2 \le c_{\gamma} (\tau) \big(\varepsilon^{\frac{1}{2} - \theta} +\varepsilon^{1 - 2(\gamma+\theta)} \big)=:\delta.
\]
Since $\mathcal{H}^2(\mathcal{S}_f), \mathcal{H}^2(\mathcal{S}_g) > \frac{1}{2},$ we conclude that $(\mathcal{S}_h,\mathcal{S}_f,\mathcal{S}_g)$ form a triple of near-extremals for the Brunn--Minkowski inequality in dimension 2. We then conclude, from either \cite{van-Hintum-Spink-Tiba-2} or \cite{Figalli-van-Hintum-Tiba}, that, denoting the closure of the convex hull of $\mathcal{S}_f,\mathcal{S}_g, \mathcal{S}_h$ by $S_f,S_g,S_h,$ respectively, there are $\tilde{w} = (w,\varrho) \in \R^{2},$ and a convex set $\mathbb{S}_h \supset S_h$ with 
\begin{equation}\begin{split}\label{eq:brunn-minkowski-optimal-2}
\mathbb{S}_h &\supset (\mathcal{S}_f - \tilde{w}) \cup (\mathcal{S}_g + \tilde{w}),\cr
\mathcal{H}^{2}(S_h \setminus \mathcal{S}_h) + \mathcal{H}^{2}(S_f \setminus \mathcal{S}_f) & + \mathcal{H}^{2}(S_g \setminus \mathcal{S}_g) \le c(\tau)  |\log\eps|^{\frac{4}{\tau}}\delta ,\cr
\mathcal{H}^{2}(\mathbb{S}_h \setminus \mathcal{S}_h) + \mathcal{H}^{2}(\mathbb{S}_h \setminus (\mathcal{S}_f - \tilde{w})) & + \mathcal{H}^{2}(\mathbb{S}_h \setminus (\mathcal{S}_g + \tilde{w})) \le  c(\tau)  |\log\eps|^{\frac{4}{\tau}}\delta^{1/2}. 
\end{split}\end{equation}
We now employ the analysis of \cite[Lemma~6.1]{Boroczky-De}. Suppose first $\tilde{w} = (w,\varrho), \, \varrho > 0.$ We let 
$$S_f^{\varrho} = \{(x,T) \in S_f \colon \theta \log \eps \le T \le \theta \log \eps + \varrho\}.$$ 
By the fact that $\mathcal{H}^{2}(S_f + (0,\varrho))  = \mathcal{H}^{2}(S_f) = \mathcal{H}^{2}(S_f \cap (S_f + (0,\varrho))) + \mathcal{H}^{2}(S_f^{\varrho})$, it follows that $\mathcal{H}^{2}(S_f \Delta (S_f + (0,\varrho))) = 2 \mathcal{H}^{2}(S_f^{\varrho}).$ Since  
$S_f^{\rho} \subset S_f \setminus (S_h+\tilde{w}),$ we also have
$$\mathcal{H}^{2}(S_f^{\varrho}) \le c(\tau) |\log\eps|^{\frac{4}{\tau}} \delta^{1/2}.$$
Thus, by triangle inequality, 
\begin{equation*}\begin{split}
\mathcal{H}^{2}(S_f \Delta (S_h + (w,0))) \le 2\mathcal{H}^{2}(S_f^{\varrho}) + \mathcal{H}^{2}(S_f \Delta (S_h + \tilde{w})) \le c(\tau) |\log\eps|^{\frac{4}{\tau}} \delta^{1/2}. 
\end{split}\end{equation*}
A similar argument works in case $\varrho < 0,$ if one considers $S^{|\varrho|}_h$ instead of $S^{\varrho}_f.$ In the end, this allows one 
to conclude that 
\begin{equation}\label{eq:close-convex-sets-dim-1}
\mathcal{H}^{2}(S_h \Delta (S_f - w)) + \mathcal{H}^{2}(S_h \Delta (S_g + w)) \le c(\tau)|\log\eps|^{\frac{4}{\tau}} \delta^{1/2}. 
\end{equation}
We now note that, as  
$\{f > \eps^{\theta}\} \times \{ T = \theta \log \eps \} \subset \mathcal{S}_f,$ then 
\[
S_f \supset \text{co}(\{f > \eps^{\theta}\}) \times \{ T= \theta \log \eps \}.
\]
We associate to each $x \in \text{co}(\{f > \eps^{\theta}\})$ the function 
$$T_f(x) = \sup\{ T \in \R \colon (x,T) \in S_f\}.$$
This function satisfies $T_f(x) \ge \theta \log \eps$ for every $x \in \text{co}(\{ f> \eps^{\theta}\}).$ Also, it is this function is \emph{concave}. Hence, we let 
$$\tilde{f}(x) = \begin{cases}
		  e^{T_f(x)}, & \text{ if } x \in \text{co}(\{ f>\eps^{\theta}\}); \cr
		  0, & \text{ otherwise }.\cr
                 \end{cases}
$$
Now notice that $(x,r)$ belongs to the interior of $S_f$ if and only if $T_f(x) > r > \theta \log \eps$ and $x$ belongs to the interior of $\text{co}(\{ f > \eps^{\theta}\}).$ This shows that 
\[
\mathcal{H}^2(S_f \setminus \mathcal{S}_f) \ge \frac{1}{2} \int_{\eps^{\theta}}^2 \mathcal{H}^1(\{ \tilde{f} > s\} \Delta \{ f > s\}) \, \d s.
\]
Moreover, we see that, by Chebyshev's inequality and \eqref{eq:brunn-minkowski-optimal-2}, there is 
$$s_0 \in (\eps^{\theta},   \eps^{\theta} + c(\tau) |\log\eps|^{2/\tau} \delta^{1/2})$$
such that
\begin{equation}\label{eq:difference-levels-f-a}
\mathcal{H}^1(\{ \tilde{f} > s_0\} \Delta \{ f > s_0\}) \lesssim c_{\gamma} (\tau)|\log\eps|^{2/\tau} \varepsilon^{1/2 - \gamma}.
\end{equation}
Define then the function $\tilde{f}_1$ to be zero whenever $\tilde{f} \le s_0,$ and equal to $\tilde{f}$ otherwise. This new function is again log-concave. We claim that it is sufficiently close to $f.$ 

Effectively, we have by the previous considerations that
\begin{equation}\begin{split}\label{eq:final-compare-tilde-f-f}
\| \tilde{f}_1 - f\|_1 &= \int_0^2 \mathcal{H}^1(\{ \tilde{f}_1 > t\} \Delta \{ f > t\}) \, \d t \cr
		     &\le \int_0^{s_0} \left(\mathcal{H}^1(\{ \tilde{f}_1 > s_0\}) + \mathcal{H}^1(\{ f > t\})\right) \, \d t 
		     + \int_{s_0}^2  \mathcal{H}^1(\{ \tilde{f}_1 > t\} \Delta \{ f > t\}) \, \d t \cr 
		     &\le c(\tau) |\log\eps|^{6/\tau} s_0 + \int_{s_0}^2 \mathcal{H}^n(\{ \tilde{f} > t\} \Delta \{ f > t\}) \, \d t \cr 
		     &\le  c(\tau) |\log\eps|^{6/\tau} s_0+ 2\mathcal{H}^{n+1}(S_f \setminus \mathcal{S}_f) \le c(\tau) |\log\eps|^{6/\tau} \varepsilon^{1/4},
\end{split}\end{equation}
where we chose $\theta = \frac{1}{4}.$ By employing the same argument for $g$ and $h,$ we are able to construct functions $\tilde{g}_1, \tilde{h}_1,$ both of which are log-concave, which satisfy 
\[
\|\tilde{g}_1-g\|_1 + \|\tilde{h}_1 - h\|_1 + \|\tilde{f}_1 -f\|_1 \le c(\tau) |\log \varepsilon|^{6/\tau} \varepsilon^{1/4}.
\]
We have, furthermore, that 
\begin{equation}\begin{split}\label{eq:log-concave-close-high-d}
c(\tau) |\log\eps|^{6/\tau} \varepsilon^{1/4} \ge \int_{\R^n} \left( |\tilde{h}_1(x) - \tilde{f}_1(x+w)| + |\tilde{h}_1(x) - \tilde{g}_1(x-w)| \right) \, \d x,
\end{split}\end{equation}
therefore
\[
\|\tilde{h}_1 - h\|_1 + \| \tilde{h}_1 (\cdot - w) - f \|_1 + \| \tilde{h}_1 (\cdot + w) - g\|_1 \le  c(\tau) |\log\eps|^{6/\tau} \varepsilon^{1/4}. 
\]
This concludes the result in the $n=1$ case, with exponent $\frac15$.

\subsection{Part II: the higher-dimensional case} In this part, we highlight the changes needed in order to prove Theorem \ref{thm:uniform-high-d} in its full generality. 
Again up to scaling and multiplication, we can assume that
 \begin{equation}
 \label{eq:normalize fg d}\|f\|_\infty =\min(\|f\|_\infty,\|g\|_\infty)=1\qquad\text{and}\qquad \|f\|_{1} = \|g\|_{1} = 1.\end{equation}
 
\vspace{2mm}

\noindent\textbf{Step 1. Estimates for the distribution functions.} We first need to recall the following estimate from \cite[Lemma~5.2]{Boroczky-Figalli-Ramos}:

\begin{proposition}[Lemma~5.2 from \cite{Boroczky-Figalli-Ramos}]\label{prop:levels-close} Let $h,f,g:\R^n \to \R$ be functions satisfying \eqref{eq:PL-condition}, \eqref{eq:PL-almost-eq}, and \eqref{eq:normalize fg d}. Let $H,F,G:\R_+ \to \R_+$ be their distribution functions. Then there is an absolute, dimensional constant $c_n(\tau)> 0$ such that 
\[
\int_0^{\infty} |F(t) - H(t)| \, \d t + \int_0^{\infty} |G(t) - H(t)| \, \d t \le c_n(\tau) \varepsilon^{\alpha_0},
\]
for some $\alpha_0 > 0.$ 
\end{proposition}

\begin{proof} Indeed, Lemma~5.2 from \cite{Boroczky-Figalli-Ramos} does not yield the result in the exact current formulation, but rather it is an \emph{automatic} difference estimate: if for any $h,f,g:\R \to \R_+$ with $\int f = \int g = 1$ which satisfy \eqref{eq:PL-condition}, \eqref{eq:PL-almost-eq}, and \eqref{eq:normalize fg} stability holds with some power $Q(\tau),$ then the conclusion of Proposition \ref{prop:levels-close} holds in the following form: 
\[
\int_0^{\infty} |F(t) - H(t)| \, \d t + \int_0^{\infty} |G(t) - H(t)| \, \d t \le c_n(\tau) \varepsilon^{Q(\tau)/2}.
\]
Since we have proved in Part I that we may take $Q(\tau) = \frac{1}{5}$ for each $\lambda \in (0,1),$ we conclude the validity of our claim with $\alpha_0 = \frac{1}{10}.$
\end{proof}

\vspace{2mm}

\noindent\textbf{Step 2. Estimates on the measure of level sets.} Before moving on to the main part of our argument, we need to prove that we can cut a fixed proportion of the functions at hand, where the height of the cut is independent in the dimension -- at least in terms of the power of $\varepsilon$ at which one cuts. 

In order to do that, we first show the following uniform bound on the $L^{\infty}$ norm of $g$, given that $\|f\|_{\infty}$ is controlled.  

\begin{lemma} Let $h,f,g:\R^n \to \R_+$ satisfy \eqref{eq:PL-condition}, \eqref{eq:PL-almost-eq}, and \eqref{eq:normalize fg d}.  Then 
\begin{equation}\label{eq:bound-size-g}
\|g\|_{\infty} \le c_n(\tau).
\end{equation} 
\end{lemma} 

\begin{proof}
Indeed, if $y_0 \in \R^n$ is fixed, we have 
\[
C_t \supset (1-\lambda) A_{t^{\frac{1}{1-\lambda}}/g(y_0)^{\frac{\lambda}{1-\lambda}}} + \lambda y_0.
\]
In particular, 
\begin{equation*}\begin{split}
\int_0^{t} F(s) \, ds & = \frac{1}{1-\lambda} \int_0^{t^{1-\lambda} g(y_0)^{\lambda}} F\left( \frac{r^{1/(1-\lambda)}}{g(y_0)^{\lambda/(1-\lambda)}}\right) \left( \frac{r}{g(y_0)} \right)^{\lambda/(1-\lambda)} \, dr \cr
 & \le \frac{1}{1-\lambda} \left( \frac{t}{g(y_0)} \right)^{\lambda} \int_0^{t^{1-\lambda} g(y_0)^{\lambda}} F\left( \frac{r^{1/(1-\lambda)}}{g(y_0)^{\lambda/(1-\lambda)}}\right) \, dr\cr
 & \le \frac{1}{(1-\lambda)^{n+1}} \left( \frac{t}{g(y_0)} \right)^{\lambda} \int_0^{t^{1-\lambda} g(y_0)^{\lambda}}  H(r) \, dr.
\end{split}\end{equation*}
Therefore, choosing $t=1$ and using that $\int H \le 1 + \varepsilon$ and $\int_0^1 F(s) \, ds = 1,$ we get
\[
g(y_0) \le \frac{2\cdot (1+\eps)^{1/\lambda} }{(1-\lambda)^{(n+1)/\lambda}} \le c_n(\tau),
\] 
as desired. 
\end{proof}

 We now formulate an important result on the size of level sets for the higher-dimensional case, which allows us to cut away the tails of log-concave functions at an uniform height. 

\begin{proposition}\label{prop:measure-levels-crucial} Let $h,f,g:\R^n \to \R_+$ be as in Proposition \ref{prop:levels-close}. Then, for each $\beta \ge \varepsilon^{\alpha_0},$ we have
\begin{equation}\label{eq:control-levels-high-d}
F(\beta) + G(\beta) \le c_n(\tau) |\log\varepsilon|^{c_n \frac{|\log \tau|}{\tau}},
\end{equation}
where $c_n>0$ is a dimensional constant. 
\end{proposition}

\begin{proof} Thanks to \cite[Equation~(5.21)]{Boroczky-Figalli-Ramos}, the conclusion of the Proposition holds as long as $\beta \ge \varepsilon^{Q(\tau)/2},$ where $Q(\tau)$ is the stability exponent obtained in dimension $1$. Again, since we have already proved that we may take $Q(\tau) = \frac{1}{5}$ for all $\tau \in (0,1/2],$ Proposition \ref{prop:measure-levels-crucial} follows at once. 
\end{proof}

\vspace{2mm}

\noindent\textbf{Step 2. Conclusion.} We are now ready to prove Theorem~\ref{thm:uniform-high-d}. In what follows, we let $c_n(\tau) > 0$ be an absolute, computable constant depending only on $n$ and $\tau$, which may change from line to line.  

Let $\theta > 0$ be small, to be chosen later. Define again the truncated log-hypographs of $f,g,h$ as 
\begin{equation*}\begin{split}
\mathcal{S}_f = \{ (x,T) \in \R^{n+1} \colon x \in \{ f > \eps^{\theta}\}, \, \eps^{\theta} \le e^T < f(x)\},\cr
\mathcal{S}_g = \{ (x,T) \in \R^{n+1} \colon x \in \{ g > \eps^{\theta}\}, \, \eps^{\theta} \le e^T < g(x)\},\cr
\mathcal{S}_h = \{ (x,T) \in \R^{n+1} \colon x \in \{ h > \eps^{\theta}\}, \, \eps^{\theta} \le e^T < h(x)\}. 
\end{split}\end{equation*}
It follows again, the same method employed in the $n=1$ case, that the measure of $\mathcal{S}_f, \mathcal{S}_g$ is well-controlled: indeed, 
\begin{equation}\label{eq:upper-bound-S_f}
c_n(\tau) \theta |\log\eps|^{c_n \frac{|\log \tau|}{\tau}} \ge  \theta |\log\eps| \cdot \mathcal{H}^n(\{f > \eps^{\theta}\}) \ge \mathcal{H}^{n+1}(\mathcal{S}_f) \ge \frac{1}{2}.
\end{equation}
The same estimates together with \eqref{eq:bound-size-g} show that 
\begin{equation}\label{eq:lower-and-upper-S_g}
c_n(\tau) \theta |\log\eps|^{c_n \frac{|\log \tau|}{\tau}} \ge \mathcal{H}^{n+1}(\mathcal{S}_g) \ge \frac{1}{c_n(\tau)} 
\end{equation}
holds as well. Employing Proposition~\ref{prop:measure-levels-crucial}, we obtain that 

\begin{equation}\begin{split}\label{eq:close-volumes-high-d}
& |\mathcal{H}^{n+1}(\mathcal{S}_f) - \mathcal{H}^{n+1}(\mathcal{S}_h)| + |\mathcal{H}^{n+1}(\mathcal{S}_g) - \mathcal{H}^{n+1}(\mathcal{S}_h)| \cr 
& \le  \int_{\theta \log \eps}^{\infty} \left( |F(e^s) - H(e^s)| +|G(e^s) - H(e^s)| \right) \, \d s \cr
& \le  \eps^{-\theta}  \int_0^{\infty} \left(|F(t) - H(t)| + |G(t) - H(t)|\right) \, \d s  \cr 
& \le c_n(\tau) \varepsilon^{\alpha_0 - \theta}=:\delta. 
\end{split}\end{equation}
By \eqref{eq:PL-condition}, it still follows that
\begin{equation}\label{eq:Brunn-Minkowski-containment-graph} 
\lambda \mathcal{S}_f + (1-\lambda)\mathcal{S}_g \subset \mathcal{S}_h.
\end{equation}
In particular, \eqref{eq:close-volumes-high-d}, \eqref{eq:Brunn-Minkowski-containment-graph}, and the fact that $\mathcal{H}^{n+1}(\mathcal{S}_f) > 1/2$, imply 
\begin{equation}\label{eq:lower-upper-s-h}
c_n(\tau) \theta |\log\eps|^{c_n \frac{|\log \tau|}{\tau}} \ge \mathcal{H}^{n+1}(\mathcal{S}_h) \ge c_n(\tau).
\end{equation}
Now, we use the main result in \cite{Figalli-van-Hintum-Tiba}. Indeed, that result states that, under the conditions satisfied by the sets $\mathcal{S}_f, \mathcal{S}_g$, and $\mathcal{S}_h$ in \eqref{eq:upper-bound-S_f}, \eqref{eq:lower-and-upper-S_g}, \eqref{eq:close-volumes-high-d}, and \eqref{eq:Brunn-Minkowski-containment-graph}, then for $\delta < d_n$,  the sets $\mathcal{S}_f,\mathcal{S}_g$ are both close (in \emph{sharp} quantitative terms of 
$\delta$) to their convex hulls. 

In more effective terms, \cite[Theorem~1.5]{Figalli-van-Hintum-Tiba} implies that there exist an absolute constant $c_n(\tau)> 0$ such that the following holds. Denote again the closure of the convex hull of $\mathcal{S}_f,\mathcal{S}_g, \mathcal{S}_h$ by $S_f,S_g,S_h$ respectively. There are $\tilde{w} = (w,\varrho) \in \R^{n+1},$ and a convex set $\mathbb{S}_h \supset S_h$ with 
\begin{equation}\begin{split}\label{eq:brunn-minkowski-optimal-n+1}
\mathbb{S}_h &\supset (\mathcal{S}_f - \tilde{w}) \cup (\mathcal{S}_g + \tilde{w}),\cr
\mathcal{H}^{n+1}(S_h \setminus \mathcal{S}_h) + \mathcal{H}^{n+1}(S_f \setminus \mathcal{S}_f) & + \mathcal{H}^{n+1}(S_g \setminus \mathcal{S}_g) \le c_n(\tau) |\log\eps|^{c_n \frac{|\log \tau|}{\tau}}\delta,\cr
\mathcal{H}^{n+1}(\mathbb{S}_h \setminus \mathcal{S}_h) + \mathcal{H}^{n+1}(\mathbb{S}_h \setminus (\mathcal{S}_f - \tilde{w})) & + \mathcal{H}^{n+1}(\mathbb{S}_h \setminus (\mathcal{S}_g + \tilde{w})) \le  c_n(\tau) |\log\eps|^{c_n \frac{|\log \tau|}{\tau}} \delta^{1/2}. 
\end{split}\end{equation}
An analysis entirely analogous to the one in Step 3 in Part I shows that 
\begin{equation}\label{eq:close-convex-sets-dim-n}
\mathcal{H}^{n+1}(S_h \Delta (S_f - w)) + \mathcal{H}^{n+1}(S_h \Delta (S_g + w)) \le c_n(\tau) |\log\eps|^{c_n \frac{|\log \tau|}{\tau}} \delta^{1/2}. 
\end{equation}
 The same method of constructing log-concave functions as before shows that, by choosing $\theta = \alpha_0/8,$ there exists $\tilde{h}_1$ log-concave such that 
\[
\| \tilde{h}_1 (\cdot - w) - f \|_1 + \| \tilde{h}_1 (\cdot + w) - g\|_1 + \|\tilde{h}_1 - h\|_1 \le  c_n(\tau) |\log\eps|^{c_n \frac{|\log \tau|}{\tau}} \varepsilon^{\alpha_0/8} \le \tilde{c}_n(\tau) \varepsilon^{\alpha_0/16}. 
\]
This concludes the proof of the claim for $n \ge 2,$ as desired. 

\begin{remark} Instead of using Proposition \ref{prop:levels-close} in order to propagate the result from dimension $1$ to higher dimensions with the aid of the sharp Brunn--Minkowski stability estimate, one could use a version of Proposition \ref{prop:f-g-dist} for radial functions in higher dimensions.

This would imply that one can take virtually the same constant as in the one-dimensional proof in higher dimensions as well. The proof of such a radial version would work in the same fashion, but with additional notational and technical obstacles. In order to keep the exposition short, we decided not to include it in this manuscript. 

However, it should be noted that even by doing so the stability exponent obtained and optimizing the methods above, we do not expect to be able to obtain the conjectured optimal exponent $\frac{1}{2}$. 
\end{remark}

\begin{remark} We have not attempted to compute the functions $c_n(\tau),$ although this is certainly achievable through a thorough inspection of the proof above. 

On the other hand, we would expect that such an attempt would yield a poor dependency on $\tau$. This stems mainly from the fact the estimates in Lemma \ref{thm:cutting-support} for measures of level sets, in spite of having a seemingly harmless $\log(1/\varepsilon)^{c_n(\tau)}$ factor, introduce exponential dependencies on $1/\tau$, since the power of the logarithm is of order $c_n/\tau$. It is possible to replace the conclusions of Lemma \ref{thm:cutting-support} and Proposition \ref{prop:measure-levels-crucial} by an \emph{absolute} power of $|\log \varepsilon|,$ but at least with the method from \cite{Boroczky-Figalli-Ramos} it seems that one still has to pay the price of a constant of the form $C^{1/\tau}, C>1,$ in front.
\end{remark}

\section{Proof of Theorems \ref{thm:log-conc-sharp} and \ref{thm:radial-sharp}} 

\subsection{Proof of Theorem \ref{thm:log-conc-sharp} for $\lambda = 1/2$}

We may suppose $\int f = \int g = 1.$ Let $T$ be the transport map taking $g$ to $f,$ in the sense of Proposition \ref{prop:f-g-transp}. By translating $f$, we may assume that $f$ has a global maximum at 0, and by translating $g$ if needed, we may assume that $T(0) = 0.$ After those reductions, we wish to prove that there is a common log-concave function $\tilde{h}$ such that 
\[
\int_{\R} |f(x) - \tilde{h}(x)| \, \d x + \int_{\R} |g(x) - \tilde{h}(x)| \, \d x + \int_{\R} |h(x)-\tilde{h}(x)| \, \d x \lesssim \varepsilon^{1/2}.
\]
We then make the following basic observation, already present in \cite{Boroczky-Ball-1,Boroczky-Ball-2} and in the proof of Proposition \ref{prop:f-g-dist} above: since $h$ satisfies that $ h\left( \frac{x+y}{2}\right) \ge \sqrt{f(x)g(y)}$ for every $x,y \in \R,$  we may write 
\begin{align}\label{eq:proof-PL}
1 & = \int_{\R} f(x) \, \d x = \int_{\R} \sqrt{f(x) g(T(x)) T'(x)} \, \d x \cr 
  & \le \int_{\R} h\left( \frac{x+T(x)}{2}\right) \sqrt{T'(x)} \, \d x \cr 
  & \le  \int_{\R} h\left( \frac{x+T(x)}{2} \right) \frac{1+T'(x)}{2} \, \d x = \int_{\R} h(y) \, \d y. 
\end{align}
Since we know that $\int_{\R} h \le 1+\varepsilon,$ from the chain of inequalities above we get
\begin{align*} 
\varepsilon & \ge \int_{\R} h\left( \frac{x+T(x)}{2} \right) \left( \frac{1+T'(x)}{2} - \sqrt{T'(x)} \right) \, \d x \cr 
            & \ge \frac{1}{2} \int_{\R} \frac{f(x)}{\sqrt{T'(x)}} \left(1-\sqrt{T'(x)}\right)^2 \, \d x  \ge \int_{\R} f(x) \frac{(1-\sqrt{T'(x)})^2}{2\sqrt{T'(x)}} \, \d x.
\end{align*} 
The inequality we obtain, that is, 
\begin{equation}\label{eq:first-eps-bound}
\varepsilon \ge \int_{\R} f(x) \frac{(1-\sqrt{T'(x)})^2}{2\sqrt{T'(x)}} \, \d x,
\end{equation}
will be the main tool in our proof of sharp stability.

Before proving Theorem \ref{thm:log-conc-sharp},
in the following lemma
we use an argument from \cite{Boroczky-Ball-2} to show that it suffices to prove that $f$ and $g$ are close.  

\begin{lemma}\label{lemma:close-f-g-h} Suppose that, for $f,g,h$ as in the statement of Theorem \ref{thm:log-conc-sharp}, normalized so that $\int f = \int g =1,$ we have that 
\[
\int_{\R} |f(x) - g(x)|\, \d x  \le C \varepsilon^{1/2}, 
\]
for some absolute constant $C>0.$ Then there is an absolute constant $\tilde{C} > 0$ such that 
\begin{equation}\label{eq:h-f-close}
    \int_{\R} |h(x) - f(x)| \, \d x \le \tilde{C} \varepsilon^{1/2}. 
\end{equation}
\end{lemma}

\begin{proof} We suppose first that $f,g$ are both log-concave. Let then $A_f := \supp(f), A_g := \supp(g).$ We start by defining the auxiliary function $\tilde{h}:\R \to \R_+$ by $\tilde{h}(x) = 0$ if $x \not\in \frac{1}{2}\left( A_f + A_g\right),$ and by 
\[
\tilde{h}\left( \frac{x+T(x)}{2} \right) = \sqrt{f(x) g(T(x))}
\]
otherwise. It then follows at once from the definitions that $h \ge \tilde{h}$ pointwise. Furthermore, note that \eqref{eq:proof-PL} applies verbatim by replacing $h$ by $\tilde{h}$. Hence, 
\[
\int|h(x) - \tilde{h}(x)| \, \d x = \int_{\R} (h(x) - \tilde{h}(x)) \, \d x \le \varepsilon.
\]
In order to prove \eqref{eq:h-f-close}, we only have to prove that $\tilde{h}$ is $\varepsilon-$close to $f.$ In order to do that, note that we instantly have 
\[
\int_{\R} \left( \tilde{h}(x) - g(x) \right) \, \d x \le \varepsilon. 
\]
Thus, let $\tilde{A} = \{ x \in \R \colon \tilde{h}(x) > g(x)\}$, and $B$ its counter-image under $\frac{x+T(x)}2$. Note that $\tilde{A} \subset \frac{1}{2} \left( A_f + A_g\right).$ Moreover, we have 
\begin{align*}
\int_{\tilde{A}} \left( \tilde{h}(x) - g(x) \right) \, \d x & = \int_{B} \left( \tilde{h}\left( \frac{x+T(x)}{2}\right) - g\left( \frac{x+T(x)}{2} \right) \right) \cdot \frac{1+T'(x)}{2} \, \, \d x \cr 
 & \le \int_{B} \left( \sqrt{f(x) g(T(x))} - \sqrt{g(x)g(T(x))}\right) \cdot \frac{1+T'(x)}{2} \, \d x \cr 
 & \le \int_{B} \left(f(x) - g(x)\right) \frac{\sqrt{g(T(x))}}{\sqrt{f(x)}} \cdot \frac{1+T'(x)}{2} \, \d x \cr 
 & = \int_{B} \left(f(x) - g(x)\right) \left( 1 + \frac{(1-\sqrt{T'(x)})^2}{2\sqrt{T'(x)}} \right) \, \d x \cr 
 & \le \int_{\R} |f(x) - g(x)| \, \d x + \int_{\R} f(x) \frac{(1-\sqrt{T'(x)})^2}{2\sqrt{T'(x)}}  \, \d x \cr 
 & \le \int_{\R} |f(x) - g(x)| \, \d x + \varepsilon. 
\end{align*}    
Here, we used the change of variables $x \mapsto \frac{x+T(x)}{2}$, taking $\tilde{A}$ to its image under the inverse of that change-of-variables map in the first passage; then we used the definition of $\tilde{h}$ and the log-concavity of $g$ in the second line, and from that point on we simply used the previous considerations and the hypothesis of the lemma. 

Now, for the case where $h$ is log-concave, we consider $\tilde{f},\tilde{g}$ the log-concave hulls of $f,g,$ respectively. Then it follows that, since $h$ is log concave, 
\[
h\left( \frac{x+y}{2} \right) \ge \sqrt{\tilde{f}(x) \tilde{g}(y)}, \qquad \forall \, x,y \in \R. 
\]
Since $\int_{\R} h \le 1+\varepsilon,$ and since, by definition, we have $\tilde{f} \ge f, \tilde{g} \ge g$ pointwise, then 
\[
\int_{\R} |f(x) - \tilde{f}(x)| \, \d x + \int_{\R} |g(x) - \tilde{g}(x)| \, \d x \le \varepsilon.
\]
Upon normalizing $(\tilde{f},\tilde{g}) \mapsto (\mathfrak{f},\mathfrak{g}):= \left( \frac{\tilde{f}}{\int_{\R} \tilde{f}}, \frac{\tilde{g}}{\int_{\R} \tilde{g}}\right),$ we have that $h,\mathfrak{f},\mathfrak{g}$ satisfy \eqref{eq:PL-condition} and \eqref{eq:PL-almost-eq}. It follows that 
\begin{equation}
\label{eq:f mathfrak}
\biggl| \int_{\R} |\mathfrak{f}(x) - \mathfrak{g}(x)| \, \d x - \int_{\R} |f(x) - g(x)| \, \d x\biggr| \lesssim \varepsilon.
\end{equation}
Hence, the conclusion in this case follows from the first case we treated, finishing thus the proof.
\end{proof}

We can now start with the proof of Theorem \ref{thm:log-conc-sharp}.
We shall first consider the case when $f,g$ are log-concave, and then the case when $h$ is log-concave.

\vspace{2mm}

\noindent
$\bullet$ {\bf Case 1: $f,g$ are log-concave.}
We begin with the following result.

\begin{lemma}\label{lemma:bounds} Let $x_1,x_2,y_1,y_2$ be defined such that $y_i = T(x_i), \, i=1,2,$ and 
    \[
    \int_{-\infty}^{x_1} f = \int_{x_2}^{\infty} f = 8 \varepsilon < 1/6.
    \]
    Assume that $f$ and $g$ are log-concave.
    Then, for $x\in(x_1,x_2)$ we have $T'(x) < 16,$ and for $y \in (y_1,y_2)$ we have $S'(y) < 16.$ 
\end{lemma} 

    \begin{proof}[Proof of Lemma \ref{lemma:bounds}] We suppose, for the sake of a contradiction, that there is $x \in (x_1,x_2)$ with $T'(x) \ge 16.$ Suppose without loss of generality that $\int_x^{\infty} f = \nu < \frac{1}{2}.$ This implies, by Proposition \ref{prop:f-g-transp}, that $\frac{f(x)}{16} \ge g(T(x)).$ From Proposition \ref{prop:log-conc}, part (i), we have that $f(t) > \frac{f(x)}{2} $ for $t \in (x,x+\log(2) \nu/f(x)).$ On the other hand, Proposition \ref{prop:log-conc}, part (ii) implies that $g(T(t)) < 2 g(T(x))$ for $t \in (x,x+\log(2) \nu/f(x)).$ Hence, we obtain that 
\[
g(T(t)) < 2 g(T(x)) \le \frac{f(x)}{8} \le \frac{f(t)}{4} = \frac{g(T(t)) T'(t)}{4}.
\]
Hence, $T'(t) > 4$ for such $t,$ and thus 
\[
\varepsilon \ge \int_{\R} \frac{(1-\sqrt{T'(t)})^2}{2\sqrt{T'(t)}}f(t) \, \d t > \frac{1}{4} \frac{f(x)}{2} \frac{\log(2) \nu}{f(x)} = \frac{\log(2)}{8} \nu. 
\]
Since $x \in (x_1,x_2),$ we must have $\nu \ge 8 \varepsilon,$ a contradiction. Thus, $T'(x) < 16,$ as desired. 

The estimate for $S'$ is analogous.
\end{proof} 

Consider then $f_1 = 1_{(x_1,x_2)} f, g_1 = 1_{(y_1,y_2)} g.$ We notice that the transport map $T_1$ between $g_1$ and $f_1$ coincides with $T$ on $(x_1,x_2).$ Moreover, we have that $T'$ is bounded from above and from below by absolute constants in the interval $(x_1,x_2).$

We now have all the ingredients needed for our proof: by Lemma \ref{lemma:close-f-g-h}, we only need to show that $\int |f-g| \lesssim \sqrt{\varepsilon}.$ To do that, we only need to prove, in turn, that
	\[
	\int|f_1 - g_1| \lesssim \sqrt{\varepsilon},
	\]
 since $f_1,g_1$ are defined by cutting off a tail of size at most $C\varepsilon$ of $f,g$, respectively. In order to do it, we shall resort to the same overall strategy of proof of Proposition \ref{prop:f-g-dist}, but now made even \emph{more precise}, thanks to Lemma \ref{lemma:bounds}. 
	
 Let then $S$ be the inverse map of $T.$ Thus, we may write 
	\begin{align*}
	\int|f_1(x) - g_1(x)| \, \d x &= \int |f_1(x) - f_1(S(x)) \cdot S'(x)| \, \d x \cr 
	 				    &\le \int |f_1(x) - f_1 \circ S(x)| \, \d x + \int |f_1 (S(x))| |S'(x) - 1| \, \d x \cr 
	 				    & = \int |f_1(x) - f_1 \circ S(x)| \, \d x + \int |f_1(y)| |T'(y) - 1| \, \d y.
	\end{align*}
   We now claim that 
   \[
   \int |f_1(x) - f_1 \circ S(x)|\, \d x \le 2 \int_{\R} |f_1(s)| |T'(s) - 1| \, \d s.
   \]
   Indeed, we may write 
   \begin{align}
      & \int |f_1(x) - f_1 \circ S(x)|\, \d x  = \int_{\{x \colon S(x) > x\}}  |f_1(x) - f_1 \circ S(x)|\, \d x \cr 
      & +  \int_{\{x \colon S(x) < x\}} |f_1(x) - f_1 \circ S(x)|\, \d x =: I_1 + I_2. 
   \end{align}
   The desired bound will follow by bounding both $I_1$ and $I_2$ by the asserted quantity. Since the treatment of both terms is the same,  we will only deal with the first.  We have then that 
   
  \begin{align*}
  \int_{\{x \colon S(x) > x\}} |f_1(x) - f_1 \circ S(x)| \, \d x & = \int_{\{x \colon S(x) > x\}} \left| \int_{x}^{S(x)}  \, \d f_1(s) \right| \, \d x \cr 
  & \le \int_{\R} \left( \int_{\R} 1_{\{ x \colon S(x) > s> x \}}  \d x \right) \, |\d f_1|(s) \cr 
   & = \int_{\R} \left( \int_{\{ x \colon s > x > T(s) \}} \d x \right) \, |\d f_1| (s)  \cr
   & \le \int_{\R} |T(s) - s| \, |\d f_1|(s) \le \int_{\R} |f_1(s)| |T'(s) - 1| \, \d s.  \cr
  \end{align*} 
Here, we have used the fundamental theorem of calculus for $f_1$ in the first equality, Fubini's theorem in the passage from the first line to the second, and Proposition \ref{prop:trace-log-conc} in the last inequality. Thus, we obtain that 
\[
\int |f_1(x) - g_1(x)| \, \d x \le 3 \int_{\R} |f_1(s)| |T'(s) - 1| \, ds. 
\]
Since Lemma \ref{lemma:bounds} implies that $T' \le 16$ in the support of $f_1,$ we have from \eqref{eq:first-eps-bound} and the Cauchy-Schwarz inequality that 
\begin{align*}
\int f_1(s) |1-T'(s)| \, \d s & \le  \left(\int f_1(s) |1-T'(s)|^2 \, \d s \right)^{1/2} \left( \int f_1(s)^2 \, \d s \right)^{1/2}  \cr  
   									& \le 15 \cdot \left( \int f_1(x) \frac{(1-\sqrt{T'(x)})^2}{2 \sqrt{T'(x)}} \, \d x \right)^{1/2} \le 15 \varepsilon^{\frac{1}{2}}. 
\end{align*}
Thus, we have from the previous considerations that 
\[
\int|f(x)-g(x)| \, \d x \le 32 \varepsilon + \int|f_1(x) - g_1(x)|\, \d x \le 32 \varepsilon + 45 \varepsilon^{1/2}.
\]
This, together with Lemma \ref{lemma:close-f-g-h}, concludes our proof of Theorem \ref{thm:log-conc-sharp} when $f$ and $g$ are log-concave. 

\vspace{2mm}

\noindent
$\bullet$ {\bf Case 2: $h$ is log-concave.}
Consider the log-concave functions $\mathfrak{f}$ and $\mathfrak{g}$ defined in the proof of Lemma~\ref{lemma:close-f-g-h}. Applying Case 1 to these functions, we deduce that 
\[
\int_{\R} |\mathfrak f(x) - \mathfrak g(x)|\, \d x  \le C \varepsilon^{1/2}.
\]
Combining this bound with \eqref{eq:f mathfrak} and Lemma \ref{lemma:close-f-g-h}, we conclude the proof.

\subsection{Proof of Theorem \ref{thm:radial-sharp} for $\lambda = 1/2$.} 
Let now $h,f,g:\R^n \to \R$ be radial functions satisfying $h\left( \frac{x+y}{2} \right) \ge \sqrt{f(x)g(y)}$ for all $x,y \in \R^n.$ Suppose, moreover, that
\[
\int_{\R^n} h(x) \, \d x \le (1+\varepsilon) \left( \int_{\R^n} f(x) \, \d x \right)^{1/2} \left( \int_{\R^n} g(x) \, \d x\right)^{1/2},
\]
and let $T:\R^n\to\R^n$ denote the transport map between the probability measures $f(x) \, \d x$ and $g(x) \, \d x.$  Since $f$ and $g$ are radial,  $T$ has the form 
$$
T(x) = \mathcal{T}(|x|) \cdot \frac{x}{|x|}.
$$
Thus, since we have that $f(x) = g(T(x))\cdot \det (DT(x)),$ a straightforward computation shows that 
\[
f(|x|) = g(\mathcal{T}(|x|)) \cdot \frac{\mathcal{T}'(|x|) \mathcal{T}(|x|)^{n-1}}{|x|^{n-1}},\qquad \forall x \in \R^n.
\]
Note that $\T$ is then nothing but the transport map between the measures $f(r) r^{n-1} \, \d r$ and $g(r) r^{n-1} \, \d r.$ As those densities are log-concave as long as $f$ and $g$ are log-concave, it follows that, for $x_0,y_0 \in \R_+, \, y_0 = \T(x_0),$ such that 
\[
\int_{x_0}^{\infty} f(r) r^{n-1} \, \d r = \int_{y_0}^{\infty} g(r) r^{n-1} \, \d r =  \tilde{c}_n \varepsilon < 1/6,
\]
we have that $\T'(r) \in (1/16,16).$ Hence, we may suppose this in the upcoming steps of our proof. 

Our next step is, once more, to redo the optimal transport proof of the Pr\'ekopa--Leindler inequality: indeed, if we suppose that $\int f = \int g = 1,$ we have that 
\begin{align*}
1 & = \int_{\R^n} f(x) \, \d x = c_n \int_{0}^{\infty} \sqrt{f(r) g(\T(r)) \T'(r) \T(r)^{n-1} r^{n-1}} \, \d r \cr 
  & \le c_n \int_0^{\infty} h\left( \frac{r+\T(r)}{2}\right) \sqrt{\frac{\T'(r) \T(r)^{n-1}}{r^{n-1}}} \, r^{n-1} \, \d r \cr 
  & \le \int_0^{\infty}  h\left( \frac{r+\T(r)}{2}\right) \left(\frac{\T(r)+r}{2r}\right)^{n-1} \frac{1+\T'(r)}{2} \, r^{n-1} \, \d r = \int_0^{\infty} h(x) \, \d x. 
\end{align*}
Arguing exactly as in the proof of Theorem \ref{thm:log-conc-sharp}, we obtain that 
\begin{align*}
\varepsilon & \ge c_n \int_0^{\infty}  h\left( \frac{r+\T(r)}{2}\right) \left( \left(\frac{\T(r)+r}{2r}\right)^{n-1} \frac{1+\T'(r)}{2} - \sqrt{\frac{\T'(r) \T(r)^{n-1}}{r^{n-1}}} \right) \, r^{n-1} \, \d r \cr 
            & \ge  c_n \int_0^{\infty} \frac{f(r)}{\sqrt{\frac{\T'(r) \T(r)^{n-1}}{r^{n-1}}}}  \left( \left(\frac{\T(r)+r}{2r}\right)^{n-1} \frac{1+\T'(r)}{2} - \sqrt{\frac{\T'(r) \T(r)^{n-1}}{r^{n-1}}} \right) \, r^{n-1} \, \d r. 
\end{align*}
In order to move on with the proof, we need to use the following claim:

\begin{lemma}\label{eq:lemma-square} Let $a,b \in (1/16,16).$ There exists an absolute, dimensional constant $c_n > 0$ such that
\begin{equation}\label{eq:difference-lemma}
\left( \frac{a+1}{2}\right)^{n-1} \frac{1+b}{2} - \sqrt{b a^{n-1}} \ge c_n \left( \sqrt{b a^{n-1} } - 1\right)^2.
\end{equation}
\end{lemma}

\begin{proof} First, note that, if either $a$ or $b$ is not within, say, $\delta_0$ of $1$ -- where $\delta_0$ will be a small number, to be specified later --, then the left-hand side of \eqref{eq:difference-lemma} is larger than a fixed constant. Indeed, this follows directly by the fact that we have $\left(\frac{a+1}{2}\right)^{n-1} - a^{\frac{n-1}{2}} \ge c_n(a-1)^2$ and $\frac{b+1}{2} -\sqrt{b} \ge c_n(\sqrt{b}-1)^2.$ Hence, we may concentrate on the case in which the pair $(a,b)$ belongs to a $\delta_0-$neighborhood of $(1,1),$ viewed as elements of $\R^2.$ 

In that case, rewrite the left-hand side of \eqref{eq:difference-lemma} as 
\[
\sqrt{ba^{n-1}} \left( \left( \frac{\frac{1}{\sqrt{a}}+\sqrt{a}}{2}\right)^{n-1} \left( \frac{\frac{1}{\sqrt{b}} + \sqrt{b}}{2}\right) - 1 \right). 
\]
Since $1 \sim \sqrt{a^{n-1} b}$ for the asserted range, we need only to deal with the expression within the brackets. Let then $F(s,t) = \left( \frac{t^{-1} + t}{2} \right)^{n-1} \frac{s^{-1}+s}{2}.$ Then $\partial_s F(1,1) = \partial_t F(1,1) = 0.$ On the other hand, a direct computation shows that, for $\delta_0$ sufficiently small, we have that 
\[
\partial_t^2 F(s,t), \partial_s^2F(s,t) \ge c_n,
\]
whenever $(s,t)$ belongs to a $C\delta_0-$neighborhood of $(1,1).$ Since we have that 
\[
|\partial_s \partial_t F(s,t)| \lesssim \delta_0
\]
for $(s,t)$ in the same neighborhood, we conclude, by Taylor's theorem with integral remainder, that for such $(s,t)$, we have 
\[
F(s,t) - F(1,1) \ge c_n\left( (t-1)^2 + (s-1)^2\right).
\]
On the other hand, it also holds for such pairs $(s,t)$ that 
\[
\left(t^{n-1}s - 1\right)^2 \lesssim (t-1)^2 + (s-1)^2.
\]
The lemma then follows directly by applying these observations to $s = \sqrt{a}, t = \sqrt{b}.$
\end{proof}

We use Lemma \ref{eq:lemma-square} with $a = \frac{\T(r)}{r}$ and $b = \T'(r).$ We conclude that there exists $c_n > 0$ absolute dimensional constant such that 
\[
\left( \left(\frac{\T(r)+r}{2r}\right)^{n-1} \frac{1+\T'(r)}{2} - \sqrt{\frac{\T'(r) \T(r)^{n-1}}{r^{n-1}}} \right) \ge c_n \left( \sqrt{\frac{\T'(r) \T(r)^{n-1}}{r^{n-1}}} - 1 \right)^2. 
\]
Since $ \sqrt{\frac{\T'(r) \T(r)^{n-1}}{r^{n-1}}} \ge \frac{1}{16},$ we have that 
\begin{equation}
c_n \varepsilon \ge \int_0^{\infty} f(r) \left( \sqrt{\frac{\T'(r) \T(r)^{n-1}}{r^{n-1}}} - 1 \right)^2 \, r^{n-1} \, \d r. 
\end{equation}
We then proceed to estimate the $L^1$ distance between $f$ and $g.$ For $\Ss := \T^{-1},$ we have: 
\begin{align*}
	\int_{\R^n} |f(x) - g(x)| \, \d x &= \int_0^{\infty} \left|f(r) - f(\Ss(r)) \frac{\Ss(r)^{n-1} \Ss'(r)}{r^{n-1}} \right|  \, r^{n-1} \, \d r \cr 
	 				    &\le \int |f(r) - f \circ \Ss(r)| \, r^{n-1} \, \d r + \int |f(s)| |\T'(s) \T(s)^{n-1}  - s^{n-1}|  \, \d s. \cr 
	\end{align*}
We then bound the first term again in analogy to the previous section: in the set where $\Ss(r) > r,$ we have 
\begin{align}\label{eq:Fubini-trick}
\int_{\{\Ss(r) > r\}} |f(r) - f \circ \Ss(r)| r^{n-1} \, \d r  & \le \int_0^{\infty} |f'(r)| |\T(r)^n - r^n| \, \d r \cr 
 & \le \int_0^{\infty} |f(r)| |\T(r)^{n-1} \T'(r) - r^{n-1}| \, \d r. 
\end{align}
In the last inequality, we simply used an integration by parts, as $f$ radial and log-concave implies that $f'$ is decreasing when viewed as a function in $\R_+.$ Moreover, it is evident from the definition of $\T$ that $\T(0) = 0.$ The treatment for the set $\{\Ss(r) < r\}$ is entirely analogous. 

We conclude that it suffices to bound the term on the right-hand side of \eqref{eq:Fubini-trick}. On the other hand, since $\T' \in (1/16,16), \T(r)/r \in (1/16,16),$ we have 
\begin{align*}
\int_0^{\infty} f(r) |\T(r)^{n-1} \T'(r) - r^{n-1}| \, \d r & \le c_n \left( \int_0^{\infty} f(r)r^{n-1} \left| \frac{\T(r)^{n-1} \T'(r)}{r^{n-1}} - 1\right|^2 \, \d r \right)^{1/2} \cr 
 & \lesssim_n \left(\int_0^{\infty} f(r) \left( \sqrt{\frac{\T'(r) \T(r)^{n-1}}{r^{n-1}}} - 1 \right)^2 \, r^{n-1} \, \d r\right)^{1/2} \cr 
 & \lesssim_n \varepsilon^{1/2}. 
\end{align*}
By using the same methods as in the one-dimensional case, we conclude the proof of Theorem \ref{thm:radial-sharp}. 

\subsection{Proof of Theorems \ref{thm:log-conc-sharp} and \ref{thm:radial-sharp} for general $\lambda \in (0,1)$} We now use an argument, originally by K. B\"or\"oczky and A. De, in order to pass from the case $\lambda = 1/2$ to general $\lambda \in (0,1)$ in Theorems \ref{thm:log-conc-sharp} and \ref{thm:radial-sharp}. 

First we note that, by an argument similar to that of Lemma \ref{lemma:close-f-g-h}, we may suppose without loss of generality that \emph{all} functions involved are log-concave, by possibly passing to log-concave hulls. Furthermore, by scaling, we may assume that 
$$ 
\int_{\R} f(x) \, \d x = \int_{\R^n} g(x) \, \d x = 1.
$$
Consider then, for each $t \in (0,1)$, the log-concave function 
\begin{equation}
h_t(z) = \sup_{z = tx + (1-t)y} f(x)^{t} g(y)^{1-t}. 
\end{equation}
As we see from the result below, this new function also has an additional \emph{log-concavity} property in terms of integrals. We refer to \cite[Lemma~7.3]{Boroczky-De} for a proof. 

\begin{lemma}[Lemma~7.3 in \cite{Boroczky-De}]\label{lemma:integral-log-conc} 
Under the hypotheses above, the function $t \mapsto \int_{\R^n} h_t(x) \, \d x$ is \emph{log-concave} in $t \in [0,1]$. 
\end{lemma}

Note then that the function 
\[
\varphi(t) = \int_{\R^n} h_t(x) \, \d x 
\]
satisfies $\varphi(0) = \varphi(1) = 1$ by the normalization we adopted for $f$ and $g$, and, by the conditions on $f,g,h$, we conclude that 
\begin{align*}
1 \le \int_{\R^n} h_{\lambda}(x) \, \d x \le \int_{\R^n} h(x) \, \d x \le 1+\varepsilon. 
\end{align*}
Hence, Proposition \ref{prop:interpolation-log-concave} together with Lemma \ref{lemma:integral-log-conc} yield that 
\begin{equation}\label{eq:1/2-stab} 
\varphi\left( \frac{1}{2} \right) = \int_{\R^n} h_{1/2} (x) \, \d x \le 1 + \frac{\varepsilon}{\tau}. 
\end{equation} 
We then divide into two cases: 

\vspace{2mm}

\noindent{\bf $\bullet$  Case 1: the case $\lambda \in (0,1)$ of Theorem \ref{thm:log-conc-sharp}.} From the $\lambda = 1/2$ case of Theorem \ref{thm:log-conc-sharp}, we conclude from \eqref{eq:1/2-stab} that, for some $w \in \R$, 
\begin{align}
\int_{\R} |f(x) - h_{1/2}(x+w)| \, \d x &\le C \left( \frac{\varepsilon}{\tau}\right)^{1/2}, \cr 
\int_{\R} |g(x) - h_{1/2}(x-w)| \, \d x & \le C \left( \frac{\varepsilon}{\tau}\right)^{1/2}. 
\end{align}
Continuing arguing in a way similar to that in Lemma \ref{lemma:close-f-g-h}, we may also conclude that 
\[
\int_{\R} |h(x) - h_{1/2}(x)| \, \d x \le C \left( \frac{\varepsilon}{\tau}\right)^{1/2},
\]
as desired, concluding the proof of Theorem \ref{thm:log-conc-sharp}. 

\vspace{2mm}

\noindent{\bf $\bullet$  Case 2: the case $\lambda \in (0,1)$ case of Theorem \ref{thm:radial-sharp}.} We use again \eqref{eq:1/2-stab}. This, together with the $\lambda = 1/2$ case of Theorem \ref{thm:radial-sharp} already proved, shows that  
\begin{align}
\int_{\R^n} |f(x) - h_{1/2}(x)| \, \d x &\le C_n \left( \frac{\varepsilon}{\tau}\right)^{1/2}, \cr 
\int_{\R^n} |g(x) - h_{1/2}(x)| \, \d x & \le C_n \left( \frac{\varepsilon}{\tau}\right)^{1/2}. 
\end{align}
We then argue one final time as in Lemma \ref{lemma:close-f-g-h}, which allows us to conclude that 
\[
\int_{\R^n} |h(x) - h_{1/2}(x)| \, \d x \le C_n \left( \frac{\varepsilon}{\tau}\right)^{1/2},
\]
finishing the proof of Theorem \ref{thm:radial-sharp} and concluding this part. 

\subsection{Counterexample construction} We now adapt the construction from \cite[Example~1.8]{Boroczky-Figalli-Ramos} in order to show that the assertion of \eqref{eq:almost-eq-conclusion} is sharp: 

\begin{proposition}\label{prop:optimality} There is an absolute constant $c \in(0,1)$ such that the following holds. For any $\varepsilon$ sufficiently small, and any $t \in (0,1)$, there exist log-concave probability densities $f, g$ on $\mathbb{R}$ such that

$$
\int_{\mathbb{R}} \sup _{z=tx + (1-t)y} f(x)^t g(y)^{1-t} \,  \d z<1+\varepsilon,
$$

while

$$
\int_{\mathbb{R}}|g(x)-f(x+x_0)|\,  \d x \geq c \left(\frac{\varepsilon}{\tau}\right)^{\frac{1}{2}} \quad \text { for any } x_0 \in \mathbb{R}.
$$   
\end{proposition}

\begin{proof} We suppose, without loss of generality, that $\tau = t \in (0,1/2).$ We fix $f(x)=e^{-\pi x^{2}},$ and let $\phi$ be an odd $C^{2}$ on $\mathbb{R}$ satisfying $\operatorname{supp} \phi \subset[-1,1]$ and $\max \phi=1$. Note that, since $\phi$ is odd, $\int_{\mathbb{R}} f \phi=0$.

For $\delta$ sufficiently small, to be fixed later, we consider $g=(1+\delta \phi) f,$ so that $\int_{\mathbb{R}} g=1$. We note that there exists an absolute constant $c>0$ such that
\begin{align}\label{eq:derivatives}
\left|\frac{\d}{\d x} [\log (1+\delta \phi)] (x) \right| & =\left|\delta \cdot \frac{\phi^{\prime}}{1+\delta \phi}\right| \leq c \delta, \cr 
\left|\frac{\d^2}{\d x^2} [\log (1+\delta \phi)] (x) \right| & =\left|\delta \cdot \frac{\phi^{\prime \prime}(1+\delta \phi)-\delta\left(\phi^{\prime}\right)^{2}}{(1+\delta \phi)^{2}}\right| \leq c \delta,
\end{align}
for any $\delta \in\left(0, \frac{1}{2}\right)$. In particular, Since $(\log f)^{\prime \prime}=-2 \pi$, it follows that $g$ is $\log$-concave for small enough $\delta$.

Note now that, since $g(x)=f(x)=e^{-\pi x^{2}}$ for $|x| \geq 1$, there exists a constant $c_{0}>0$ such that

\begin{equation}\label{eq:bound-translation}
\int_{\mathbb{R}}|g(x)-f(x+w)| \, \d x \geq \int_{1}^{\infty}\left|e^{-\pi x^{2}}-e^{-\pi(x+w)^{2}}\right| \, \d x \geq c_{0} \min \{|w|, 1\}.
\end{equation}
On the other hand, we have
$$
\int_{\mathbb{R}}|g(x)-f(x+w)| \, \d x \geq \int_{\mathbb{R}}|g(x)-f(x)|-|f(x)-f(x+w)| \, \d x \geq \delta \int_{\mathbb{R}} f(x)|\varphi(x)|\, \d x-c|w| .
$$
Hence, combining this last estimate with \eqref{eq:bound-translation}, we deduce the existence of a constant $c_{1}>0$ such that
\begin{equation}\label{eq:lower-f-g}
\int_{\mathbb{R}}|g(x)-f(x+w)| \, \d x \geq c_{1} \delta \quad \forall w \in \mathbb{R} . 
\end{equation}
Finally, we estimate $\int_{\mathbb{R}} h$ for $h(z)=\sup _{z=tx+(1-t)y} f(x)^t g(y)^{1-t}$. To this aim, consider the auxiliary function $\tilde{h}(z)=f(z)^t g(z)^{1-t}$. Thanks to Hölder's inequality, this satisfies $\int_{\mathbb{R}} \tilde{h} \leq 1$. Since $f$ and $g$ are log-concave and $g(x)=f(x)$ for $|x| \geq 1$, for any $z \in \mathbb{R}$, there exists a point $y_{z} \in \mathbb{R}$ such that $h(z)=f\left(\frac{z}{t}-\frac{(1-t) y_{z}}{t} \right)^t g\left(y_{z}\right)^{1-t}$. Also, $y_{z}=z$ if $|z| \geq 1$, and $\left|y_{z}\right| \leq 1$ if $|z| \leq 1$.

We now observe that, for any $z \in \mathbb{R}$, the function 
$$\psi_{z}(y)=\log \left( f\left(\frac{z}{t}-\frac{(1-t) y}{t} \right)^t g\left(y\right)^{1-t}\right)$$
satisfies $\psi_{z}(z)=$ $\log \tilde{h}(z), \psi_{z}\left(y_{z}\right)=\log h(z)$, and $\psi_{z}$ has a maximum at $y_{z}$. Then, recalling \eqref{eq:derivatives}, we have

$$
0=\psi_{z}^{\prime}\left(y_{z}\right)=2 \pi\frac{1-t}{t} \left(z-y_{z}\right)+(1-t)[\log (1+\delta \phi)]^{\prime}\left(y_{z}\right) \quad \Rightarrow \quad\left|z-y_{z}\right| \leq c \cdot t \cdot \delta .
$$
On the other hand, we have that
\[
\left|\psi_{z}^{\prime \prime}(y)\right| \le C \frac{1-t}{t} + \delta,
\]
which implies, through a Taylor expansion argument, that
$$
\log \frac{h(z)}{\tilde{h}(z)}=\psi_{z}\left(y_{z}\right)-\psi_{z}(z) \leq c \frac{1}{t} \cdot (t\delta)^{2} = c \cdot t \delta^2 \quad \forall z \in \mathbb{R},
$$
for some constant $c>0$. We then conclude that

\begin{equation}\label{eq:final-h}
\int_{\mathbb{R}} h \leq e^{c t\delta^{2}} \int_{\mathbb{R}} \tilde{h} \leq e^{c t\delta^{2}}<1+2c t\delta^{2} \text { for } \delta \text{ sufficiently small.}
\end{equation}
Choosing $\delta \sim \varepsilon^{\frac{1}{2}},$ we have that \eqref{eq:lower-f-g} and \eqref{eq:final-h} imply the claim of the Proposition, as desired. 
\end{proof}

\begin{remark} It can be proved, just like in the one-dimensional case, that the results in Theorem \ref{thm:radial-sharp} are optimal: indeed, instead of picking $\phi$ in the construction of the one-dimensional example to be odd, take $\phi$ to be any \emph{even} function supported in $[1/2,2]$ for which $\int_{\R} f \varphi = 0,$ and let $g(r) = (1+\delta \phi(r))f(r).$ Extend then $f$ and $g$ radially to $\R^n$, denoting by $F_n$ and $G_n$, respectively, these extensions. Since $f$ and $g$ are log-concave, it follows by a simple computation that $F_n$ and $G_n$ are also log-concave. By redoing the same computations as in the one-dimensional case, we may conclude the desired analogue of Proposition \ref{prop:optimality}. We omit the details.
\end{remark} 

%\[ 
%h\left(\frac{x+y}{2}\right) \ge \sqrt{f(x)g(y)} \, \qquad \forall \, x,y \in \R^d. 
%\]
%Suppose, moreover, that $f,g$ are both log-concave, and that 
%\[
%\int_{\R^d} h \le (1+\varepsilon) \left(\int_{\R^d} f\right)^{1/2} \left(\int_{\R^d} g\right)^{1/2}. 
%\]
%Let then $F(r) = f(r)\cdot r^{d-1},\, G(r) = g(r) \cdot r^{d-1}, \, H(r) = h(r) \cdot r^{d-1}.$ If $f,g$ are log-concave, then $F,G$ are both log-concave as well. Moreover, since the function $r \mapsto r^{d-1}$ is also log-concave -- as $\left( \frac{a+b}{2}\right)^{d-1} \ge (ab)^{\frac{d-1}{2}}, \qquad \forall \, a,b \in \R_+,$ which follows from the arithmetic-geometric inequality --, then $H,F,G$ forms an admissible triple of functions for the Pr\'ekopa--Leindler inequality in dimension 1. 

%We hence want to use Theorem \ref{thm:log-conc-sharp} in order to conclude. In order to be able to do so, we

\bibliography{biblio} 
\bibliographystyle{amsplain}
\end{document}